\documentclass[11pt]{article}

\usepackage[T1]{fontenc}
\usepackage{lmodern}

\usepackage{amsmath, amssymb, amsthm}
\usepackage{mathrsfs}
\usepackage{stmaryrd}
\usepackage[margin=2.5cm]{geometry}
\usepackage{color, graphicx, float}
\usepackage{subcaption}
\usepackage{caption}
\usepackage[breakable]{tcolorbox}
\usepackage{tikz}
\usetikzlibrary{positioning, quotes}
\usepackage{bbm}

\newlength{\R}

\usepackage{etoolbox}
\makeatletter
\patchcmd{\@maketitle}{\LARGE \@title}{\LARGE\bfseries\@title}{}{}

\renewcommand{\@seccntformat}[1]{\csname the#1\endcsname.\quad}
\makeatother

\definecolor{darkblue}{rgb}{0,0,.5}

\usepackage{hyperref}
\hypersetup{
	colorlinks=true,		
	linkcolor=darkblue,		
	citecolor=darkblue,		
	urlcolor=darkblue 		
}

\makeatletter
\def\th@plain{%
	\thm@notefont{}
	\itshape 
}
\def\th@definition{%
	\thm@notefont{}
	\normalfont 
}

\renewenvironment{proof}[1][\proofname]{\par
	\normalfont
	\topsep0\p@\@plus3\p@ \trivlist
	\item[\hskip\labelsep\itshape
	#1\@addpunct{.}]\ignorespaces
}{%
	\qed\endtrivlist
}
\makeatother

\newtheorem{theorem}{Theorem}[section]
\newtheorem{lemma}[theorem]{Lemma}

\newtheorem{proposition}[theorem]{Proposition}
\newtheorem{assumption}[theorem]{Assumption}
\theoremstyle{definition}
\newtheorem{definition}[theorem]{Definition}
\theoremstyle{definition}
\newtheorem{example}[theorem]{Example}
\theoremstyle{definition}
\newtheorem{remark}[theorem]{Remark}
\theoremstyle{definition}
\newtheorem{algorithm}{Algorithm}

\usepackage[shortlabels]{enumitem}

\allowdisplaybreaks

\newcommand{\dom}{\ensuremath{\operatorname{dom}}}
\newcommand{\Fix}{\ensuremath{\operatorname{Fix}}}
\newcommand{\zer}{\ensuremath{\operatorname{zer}}}
\newcommand{\gra}{\ensuremath{\operatorname{gra}}}
\newcommand{\Id}{\ensuremath{\operatorname{Id}}}

\newcommand{\lspan}{\ensuremath{\operatorname{span}}}
\newcommand{\ran}{\ensuremath{\operatorname{ran}}}
\newcommand{\diag}{\ensuremath{\operatorname{diag}}}
\newcommand{\Deg}{\ensuremath{\operatorname{Deg}}}
\newcommand{\Inc}{\ensuremath{\operatorname{Inc}}}
\newcommand{\Lap}{\ensuremath{\operatorname{Lap}}}
\newcommand{\Adj}{\ensuremath{\operatorname{Adj}}}

\def\bOne{{\mathbbm{1}}}

\begin{document}

\title{A general approach to distributed operator splitting}

\author{
Minh N. Dao\thanks{School of Science, RMIT University, Melbourne, VIC 3000, Australia.
E-mail:~\href{href:minh.dao@rmit.edu.au}{minh.dao@rmit.edu.au}.},
~
Matthew K. Tam\thanks{School of Mathematics and Statistics, The University of Melbourne, Parkville, VIC 3010, Australia.
E-mail:~\href{href:matthew.tam@unimelb.edu.au}{matthew.tam@unimelb.edu.au}.},
~and~
Thang D. Truong\thanks{School of Science, RMIT University, Melbourne, VIC 3000, Australia.
E-mail:~\href{href:thang.tdk64@gmail.com}{thang.tdk64@gmail.com}.}
}

\date{}

\maketitle

\begin{abstract}
Splitting methods have emerged as powerful tools to address complex problems by decomposing them into smaller solvable components. In this work, we develop a general approach to forward-backward splitting methods for solving monotone inclusion problems involving both set-valued and single-valued operators, where the latter may lack cocoercivity. Our proposed approach, based on some coefficient matrices, not only encompasses several important existing algorithms but also extends to new ones, offering greater flexibility for different applications. Moreover, by appropriately selecting the coefficient matrices, the resulting algorithms can be implemented in a distributed and decentralized manner.
\end{abstract}


\noindent{\bf Mathematics Subject Classification (MSC 2020):}
47H05,  
47H10,  
65K10,  
90C30.  

\section{Introduction}

We consider the inclusion problem
\begin{align}
\label{gen_prob}
    \text{find } x \in \mathcal{H} \text{ such that } 0 \in \sum_{i=1}^n A_i x + \sum_{j=1}^p B_j x,
\end{align}
where $\mathcal{H}$ is a real Hilbert space, $A_1, \dots, A_n\colon \mathcal{H} \rightrightarrows \mathcal{H}$ are maximally monotone operators, and $B_1, \dots, B_p \colon \mathcal{H} \rightarrow \mathcal{H}$ are either cocoercive operators, or monotone and Lipschitz continuous operators. Problem \eqref{gen_prob} encompasses many important optimization models including composite optimization problems \cite{CP11}, structured saddle-point problems \cite{Roc70}, and variational inequalities \cite[Chapter 12]{RW98}. In order to solve problem \eqref{gen_prob}, \emph{splitting algorithms} use the idea of performing the computation of each operator individually instead of their direct sum. Set-valued operators $A_1, \dots, A_n$ are processed through resolvents (\emph{backward steps}), while single-valued operators $B_1, \dots, B_p$ are handled by direct computation (\emph{forward steps}). 

In the case where $p=0$, problem \eqref{gen_prob} reduces to finding a zero in the sum of finitely many maximally monotone operators that are potentially set-valued. In turn, when $n=2$, the best-known algorithm is the \emph{Douglas--Rachford algorithm} \cite{DR56,Sva11}. When $n=3$, Ryu \cite{Ryu20} introduced a splitting algorithm along with a notion of \emph{frugal resolvent splittings}. This notion refers to a class of fixed point algorithms that can be described using only vector addition, scalar multiplication, and the resolvent of each operator once per iteration. This idea was later extended in \cite{MT23} to solve inclusion problems with arbitrary $n\geq 2$, even though it does not recover Ryu algorithm when $n=3$. To address this issue, a framework for convergence of frugal resolvent splittings was developed in \cite{Tam23} which now covers both algorithms in \cite{MT23} and \cite{Ryu20}. Using a different approach, Bredies, Chenchene, and Naldi \cite{BCN24} independently proposed graph and distributed extensions of the Douglas--Rachford algorithm, which gave the same extension of Ryu algorithm to the case $n>3$ as presented in \cite{Tam23}.

Although there exist operators whose resolvents are computable, many of them do not have closed form. Therefore, forward evaluations, which can only be used by single-valued operators, are favorable as they use the operator directly. This is the case when $n=1$ and $p=1$ in \eqref{gen_prob} and $B_1$ is cocoercive. In this situation, the \emph{forward-backward algorithm} introduced by Lions and Mercier \cite{LM79} and Passty \cite{Passty79} has been widely used, while the closely related \emph{backward-forward algorithm} by Attouch, Peypouquet, and Redont \cite{APR18} can also be employed. However, one drawback of the forward-backward and backward-forward algorithms is that they require $B_1$ to be a cocoercive operator, which can be difficult to satisfy in practice. To relax the cocoercivity requirement, Tseng \cite{Tseng00} proposed an algorithm commonly called the \emph{forward-backward-forward algorithm} for the case when $B_1$ is monotone and Lipschitz continuous, at the cost of an additional forward step at each iteration. Using the same assumption as Tseng, the \emph{forward-reflected-backward algorithm} developed in \cite{MT20} 
includes a ``reflection term'' that evaluates $B_1$ at not only the current point but also the previous point. This algorithm only uses one forward evaluation per iteration instead of two as in Tseng's algorithm. In the case where $n=2$ and $p = 1$ in \eqref{gen_prob}, one can use the \emph{Davis--Yin algorithm} \cite{DY17} if $B_1$ is a cocoercive operator. When $B_1$ is monotone and Lipschitz continuous, the \emph{backward-forward-reflected-backward algorithm} or the \emph{backward-reflected-forward-backward algorithm} \cite{RT20} can be applied. Furthermore, the authors of \cite{AMTT23} proposed a \emph{distributed forward-backward algorithm} for the case where $n\geq 2$, $p =n-1$, and $B_1,\dots,B_p$ are cocoercive, as well as a second algorithm for $n\geq 3$, $p =n-2$, and $B_1,\dots,B_p$ monotone and Lipschitz continuous. 

Despite the fact that the convergence proofs of the aforementioned algorithms have similarities, their convergence analyzes are often conducted separately, leading to a case-by-case approach. To provide a more unified perspective, we introduce a general approach to distributed operator splitting and convergence analysis that encompasses a wide range of existing methods. This generalizes the distributed forward-backward algorithms in \cite{AMTT23} and, at the same time, extends the framework in \cite{Tam23} not only by incorporating single-valued operators $B_1,\dots,B_p$, but also by allowing more flexibility in parameter and coefficient matrix selection. This also covers the algorithms devised by graphs in \cite{ACL24, BCN24}, offering a more accessible and implementable alternative. Our proofs are concise and transparent, relying on the Krasnosel'ski\u{\i}--Mann iteration and the weak-strong sequential closedness with the fixed point operator being conically quasiaveraged. The conical quasiaveragedness property of the underlying operator is reduced to verifying the positive semidefiniteness of a certain matrix defined by algorithm coefficients. Furthermore, there are many existing algorithms satisfying this assumption automatically.

In this work, we also focus on \emph{distributed algorithms}, which decompose a global task into subtasks across a network of nodes with each operator assigned to a single node, and communication restricted to directly connected nodes. Of special interest are \emph{decentralized algorithms}, a subclass of distributed algorithms that operates without a central coordinator: each node maintains and updates its local variables using only local computations and communications with immediate neighbors \cite[Section 11]{RY22}. Our approach utilizes coefficient matrices that, in certain cases, can be derived from \emph{weighted graphs} representing information flow between nodes in a network. Specialized instances of our framework are then presented within this graph-based setting. First, we generalize the \emph{forward-backward algorithms devised by graphs} \cite{ACL24} to settings that may not require cocoercivity. Second, we encompass several existing schemes such as the forward-backward and forward-reflected-backward algorithms for ring networks \cite{AMTT23}, the sequential and parallel forward-Douglas--Rachford algorithms \cite{BCLN22}, the \emph{generalized forward-backward algorithm} \cite{RFP13}, and different product-space formulations of the Davis--Yin algorithm, including a reduced dimensional variant. Next, we derive an explicit algorithm based on complete graphs with or without cocoercivity assumptions, thereby extending the work in \cite{Tam23} that only considered the unit-weight case without establishing explicit formulas. Last but not least, we present a new splitting algorithm based on complete and star graphs that can be seen as a generalization of the Ryu splitting algorithm for $n$ set-valued and $p$ single-valued operators, where the latter are not necessarily cocoercive. We acknowledge the concurrent work \cite{ACGN25}, which adopts a similar approach and came to our attention during the final stages of this manuscript.

Our main contributions are as follows.
\begin{enumerate}
    \item We develop a general framework that integrates a variety of algorithms with different proofs, providing a concise and transparent convergence analysis through a new perspective. This not only covers many existing algorithms in the literature but also offers greater flexibility and broader applicability to a wider class of problems. 
    \item Under mild assumptions that are automatically satisfied in graph-based settings, we derive a class of distributed algorithms tailored to ring, sequential, star, and complete graph topologies. Furthermore, we present new explicit algorithms that allow flexible parameter selection and do not require cocoercivity of the single-valued operators.
\end{enumerate}

The remainder of this paper is structured as follows. In Section~\ref{s:prelim}, we introduce some notations and background materials on set-valued operators, single-valued operators and some useful lemmas. We present our general approach to distributed operator splitting in Section~\ref{s:framework}, with our main Algorithm~\ref{algo:full} and its convergence results in Theorem~\ref{t:cvg_wco} and Theorem~\ref{t:cvg_co}. In Section~\ref{s:graphs}, we demonstrate how to obtain graph-based algorithms from our framework, simultaneously deriving new algorithms and recovering existing ones. Section~\ref{s:num_exp} presents numerical experiments that demonstrate the generality of the proposed framework and examine its performance under various choices of the coefficient matrices and parameters.

\section{Preliminaries}
\label{s:prelim}

Throughout this work, $\mathcal{H}$ is a real Hilbert space with inner product $\langle \cdot, \cdot\rangle$ and induced norm $\|\cdot\|$. Strong convergence and weak convergence of sequences are denoted by $\rightarrow$ and $\rightharpoonup$, respectively. The set of nonnegative integers is denoted by $\mathbb{N}$ and the set of real numbers by $\mathbb{R}$. We use $A\colon \mathcal{H}\rightrightarrows \mathcal{H}$ to indicate that $A$ is a set-valued operator on $\mathcal{H}$ and use $A\colon \mathcal{H}\to \mathcal{H}$ to indicate that $A$ is a single-valued operator on $\mathcal{H}$.

Let $A$ be an operator on $\mathcal{H}$. The \emph{domain} of $A$ is $\dom A :=\{x\in \mathcal{H}: Ax\neq \varnothing\}$, the \emph{graph} of $A$ is $\gra A :=\{(x,u)\in \mathcal{H}\times \mathcal{H}: u\in Ax\}$, and the set of \emph{fixed points} of $A$ is $\Fix A :=\{x\in \mathcal{H}: x\in Ax\}$. The \emph{resolvent} of $A$ is defined by
\begin{align*}
J_A :=(\Id +A)^{-1},    
\end{align*}
where $\Id$ is the identity operator. We say that $A$ is \emph{monotone} if for all $(x, u), (y, v)\in \gra A$,
\begin{align*}
\langle x -y, u -v\rangle \geq 0
\end{align*}
and that $A$ is \emph{maximally monotone} if it is monotone and there exists no monotone operator whose graph properly contains $\gra A$.

An operator $B\colon \mathcal{H}\to \mathcal{H}$ is said to be \emph{$\ell$-Lipschitz continuous} for $\ell\in [0, +\infty)$ if, for all $x, y\in \mathcal{H}$,
\begin{align*}
\|Bx -By\| \leq \ell\|x -y\|;
\end{align*}
and \emph{$\frac{1}{\ell}$-cocoercive} for $\ell\in (0, +\infty)$ if, for all $x, y\in \mathcal{H}$, 
\begin{align*}
\langle Bx -By, x -y \rangle \geq \frac{1}{\ell}\|Bx -By\|^2.
\end{align*}
By the definition and the Cauchy--Schwarz inequality, every $\frac{1}{\ell}$-cocoercive operator is monotone and $\ell$-Lipschitz continuous. In turn, every monotone and Lipschitz continuous operator is maximally monotone \cite[Corollary 20.28]{BC17}.

We recall from \cite[Definition~2.1]{BDP22} that an operator $T\colon \mathcal{H}\to \mathcal{H}$ is \emph{conically $\rho$-averaged} if $\rho \in (0, +\infty)$ and for all $x, y\in \mathcal{H}$, 
\begin{align*}
\|Tx -Ty\|^2 +\frac{1-\rho}{\rho}\|(\Id -T)x -(\Id -T)y\|^2 \leq \|x -y\|^2.
\end{align*}
Similarly, we introduce the following notion of conical quasiaveragedness.

\begin{definition}
An operator $T\colon \mathcal{H}\to \mathcal{H}$ is said to be \emph{conically $\rho$-quasiaveraged} if $\rho\in (0, +\infty)$ and for all $x\in \mathcal{H}$ and all $y\in \Fix T$, 
\begin{align*}
\|Tx -y\|^2 +\frac{1-\rho}{\rho}\|(\Id -T)x\|^2 \leq \|x -y\|^2.
\end{align*}
\end{definition}
It is clear from the definition that every conically $\rho$-averaged operator is conically $\rho$-quasiaveraged. The following result slightly extends \cite[Theorem~5.15]{BC17} and \cite[Proposition~2.9]{BDP22}. Recall that a sequence $(z^k)_{k\in \mathbb{N}}$ is \emph{Fej\'er monotone} with respect to a nonempty set $V$ if, for all $v\in V$ and all $k\in \mathbb{N}$,
\begin{align*}
\|z^{k+1} -v\| \leq \|z^k -v\|.   
\end{align*}

\begin{proposition}[Krasnosel'ski\u{\i}--Mann iterations]
\label{p:KM}
Let $T$ be a conically $\rho$-quasiaveraged operator with $\Fix T \neq \varnothing$. Let $z^0 \in \mathcal{H}$. For each $k\in\mathbb{N}$, set
\begin{align*}
z^{k+1} = (1 -\zeta_k)z^k + \zeta_k Tz^k,
\end{align*}
where $(\zeta_k)_{k\in\mathbb{N}}$ is a sequence in $[0, 1/\rho]$ such that $\liminf_{k\to+\infty} \zeta_k(1-\rho\zeta_k) > 0$. Then the following hold:
\begin{enumerate}
\item 
\label{p:KM_Fejer}
$(z^k)_{k\in \mathbb{N}}$ is Fej\'er monotone with respect to $\Fix T$.
\item 
\label{p:KM_asymp}
$(\Id -T)z^k \to 0$ as $k\to +\infty$.
\item 
\label{p:KM_ergodic}
$\|\frac{1}{k+1}\sum_{t=0}^k (\Id -T)z^t\| = O(\frac{1}{\sqrt{k}})$ as $k\to +\infty$.
\end{enumerate}
\end{proposition}
\begin{proof}
\ref{p:KM_Fejer}: For all $v\in\Fix T$ and all $k\in\mathbb{N}$, by the conical quasiaveragedness of $T$,
\begin{align}
\label{eq:Fejer}
\| z^{k+1} -v\|^2 
&= \left\|(1 -\zeta_k)(z^k -v) +\zeta_k(Tz^k -v)\right\| \notag \\
&= (1 -\zeta_k)\|z^k -v\|^2 +\zeta_k\|Tz^k -v\|^2 -\zeta_k(1 -\zeta_k)\|z^k -Tz^k\|^2 \notag \\
&\leq (1 -\zeta_k)\|z^k -v\|^2 +\zeta_k\left(\|z^k -v\|^2 -\frac{1 -\rho}{\rho}\|(\Id -T)z^k\|^2\right) -\zeta_k(1 -\zeta_k)\|(\Id -T)z^k\|^2 \notag \\
&= \|z^k -v\|^2 -\frac{1}{\rho}\zeta_k(1 -\rho\zeta_k)\|(\Id -T)z^k\|^2.
\end{align}
Hence, $(z^k)_{k\in\mathbb{N}}$ is Fej\'er monotone with respect to $\Fix T$.

\ref{p:KM_asymp}: Telescoping \eqref{eq:Fejer} over $k\in\mathbb{N}$ yields 
\begin{align*}
\frac{1}{\rho}\sum_{k=0}^{+\infty} \zeta_k(1 -\rho\zeta_k)\|(\Id -T)z^k\|^2 \leq \|z^0 -v\|^2 < +\infty.
\end{align*}
Since $\rho >0$ and $\liminf_{k\rightarrow +\infty} \zeta_k(1-\rho\zeta_k) >0$, we obtain
\begin{align}
\label{eq:asymptotic}
c :=\sum_{k=0}^{+\infty} \|(\Id -T)z^k\|^2 < +\infty,
\end{align}
which implies that $(\Id -T)z^k \to 0$ as $k\to +\infty$.

\ref{p:KM_ergodic}: Using the convexity of the norm-squared and \eqref{eq:asymptotic}, we have that
\begin{align*}
\left\|\frac{1}{k+1}\sum_{t=0}^k (\Id -T)z^t\right\|^2 \leq \frac{1}{k+1}\sum_{t=0}^k \|(\Id -T)z^t\|^2 < \frac{c}{k+1},
\end{align*}
which completes the proof.
\end{proof}

Given a matrix $M$, its \emph{range} and \emph{kernel} are denoted by $\ran M$ and $\ker M$, respectively. Its \emph{transpose} is denote $M^{\top}$, and its \emph{Moore--Penrose pseudoinverse} is denoted $M^\dag$. To denote that $M$ is positive semidefinite, we write $M \succeq 0$.

\begin{lemma}
\label{l:matrix}
Let $M\in \mathbb{R}^{n\times m}$ and $L\in \mathbb{R}^{p\times n}$ be matrices. Then the following assertions are equivalent:
\begin{enumerate}
\item\label{l:matrix_exist}
There exists $U\in \mathbb{R}^{p\times m}$ such that $UM^{\top} =L$.
\item\label{l:matrix_ran}
$\ran L^{\top}\subseteq \ran M$.
\end{enumerate}
Moreover, if one of these assertions holds, then $U =L(M^{\top})^\dag$ is the matrix with minimal Frobenius norm satisfying $UM^{\top} =L$.
\end{lemma}
\begin{proof}
On the one hand, if $UM^{\top}=L$, then $MU^{\top}=L^{\top}$. The existence of $U$ implies that $\ran L^{\top}=\ran(MU^{\top}) \subseteq \ran M$. On the other hand, by using \cite[Section 3.1]{GS09}, if $\ran L^{\top}\subseteq \ran M$, the columns of $L^{\top}$ lie in the span of the columns of $M$. Then there exists a matrix $U\in\mathbb{R}^{p\times m}$ such that $MU^{\top}=L^{\top}$, which is equivalent to $UM^{\top}=L$. 

According to \cite[Theorem]{Pen56}, the minimum norm solution $\|U^{\top}\|_F=\|U\|_F$ (where $\|\cdot\|_F$ is the Frobenius norm) of the equation $MU^{\top} = L^{\top}$ is $U^{\top}=M^\dag L^{\top}=((M^{\top})^\dag)^{\top} L^{\top} = (L(M^{\top})^\dag)^{\top}$. As a result, $U=L(M^{\top})^\dag$ has minimal norm among those satisfying $UM^{\top}=L$, which completes the proof. 
\end{proof}

\begin{lemma}[{\cite[Section~4.1]{GS09}}]
\label{l:ran_ker}
    Given a matrix $M\in\mathbb{R}^{n \times m}$, the kernel of $M^{\top}$ and the column space of $M$ are orthogonal subspaces of $\mathbb{R}^{n}$, i.e., $(\ker M^{\top})^\perp=\ran M$. 
\end{lemma}

When the context is clear, we identify matrices $M=(M_{ij})\in\mathbb{R}^{n\times m}$ with the Kronecker product $(M\otimes \Id)\colon \mathcal{H}^m \to \mathcal{H}^n$. Specifically, for any $\mathbf{z}=(z_1,\dots, z_m)\in\mathcal{H}^m$, we write $(M\otimes \Id)\mathbf{z}$ by $M\mathbf{z}$. Let $A_1, \dots, A_n \colon\mathcal{H} \rightrightarrows \mathcal{H}$ be maximally monotone operators. For any $\mathbf{x}=(x_1, \dots, x_n)\in \mathcal{H}^n$, we define the operator $\mathbf{A}\colon \mathcal{H}^n \rightrightarrows \mathcal{H}^n$ by $\mathbf{Ax} := A_1x_1\times \dots\times A_n x_n$. It follows that $\mathbf{A}$ is also a maximally monotone operator. As a consequence, its resolvent $J_{\mathbf{A}}\colon \mathcal{H}^n \rightarrow \mathcal{H}^n$ is given by $J_{\mathbf{A}}=(J_{A_1}, \dots, J_{A_n})$. Similarly, for the single-valued operators $B_1, \dots, B_p\colon \mathcal{H} \rightarrow \mathcal{H}$, we define $\mathbf{B}: \mathcal{H}^p \rightarrow \mathcal{H}^p$ for any $\mathbf{u}=(u_1, \dots, u_p)\in \mathcal{H}^p$ by $\mathbf{Bu} := (B_1 u_1, \dots, B_p u_p)$.

\section{A general approach to distributed operator splitting}
\label{s:framework}

In this section, motivated by the framework in \cite{Tam23}, we propose Algorithm~\ref{algo:full} along with its convergence analysis for solving \eqref{gen_prob}, where $n \geq 2$ and $1 \leq p \leq n-1$. Here, even  though the number of single-valued operators differs from the number of set-valued operators by one or more, we can always derive schemes where these numbers match, i.e., $p=n$, depending on specific scenarios, as shown in Remark~\ref{r:ps}.  
 
\begin{tcolorbox}
[boxsep=0pt,left=0pt,right=0pt,top=0pt,bottom=0pt, colback=blue!10!white,colframe=blue!30!white,
    boxrule=0pt,breakable]
\begin{algorithm}
\label{algo:full}
Let $M =(M_{ij}) \in \mathbb{R}^{n \times m}$, $N =(N_{ij}) \in \mathbb{R}^{n \times n}$, $P =(P_{ij})\in\mathbb{R}^{n\times p}$, $Q =(Q_{ij})\in\mathbb{R}^{n\times p}$, and $R =(R_{ij}) \in \mathbb{R}^{p\times n}$. Let $\gamma \in (0, +\infty)$, $\delta_1, \dots, \delta_n\in (0, +\infty)$ and $(\lambda_k)_{k\in \mathbb{N}}\subset (0, +\infty)$. Let $z^0_1, \dots, z^0_{m}\in\mathcal{H}$. For each $k\in\mathbb{N}$, compute
\begin{align*}
\begin{cases}
x^k_i &=J_{\frac{\gamma}{\delta_i}A_i}\Bigg(\frac{1}{\delta_i}\Big(\sum_{j=1}^m M_{ij}z^k_j +\sum_{j=1}^n N_{ij}x^k_j - \gamma \sum_{j=1}^p (P_{ij} -Q_{ij})B_j\left(\sum_{l=1}^n R_{jl}x^k_l\right) \\ 
&\qquad\qquad\qquad - \gamma \sum_{j=1}^p Q_{ij}B_j\left(\sum_{l=1}^n P_{lj}x^k_l\right)\Big)\Bigg),\quad i \in \{1, \dots, n\}, \\
z^{k+1}_i \!\!\!\!\! &=z^k_i -\lambda_k\sum_{j=1}^n M_{ji}x^k_j,\quad i\in \{1, \dots, m\},
\end{cases}
\end{align*}
or equivalently,
\begin{align*}
\begin{cases}
\mathbf{x}^k &=J_{\gamma D^{-1}\mathbf{A}}( D^{-1}( M\mathbf{z}^k +  N\mathbf{x}^k -\gamma(P-Q)\mathbf{B}(R\mathbf{x}^k)-\gamma Q\mathbf{B}(P^{\top}\mathbf{x}^k))), \\
\mathbf{z}^{k+1}\!\!\!\!\! &=\mathbf{z}^k - \lambda_k M^{\top}\mathbf{x}^k,
\end{cases}
\end{align*}
where $D =\diag(\delta_1, \dots, \delta_n)$.
\end{algorithm}
\end{tcolorbox}

\begin{remark}
\label{r:OnAlgo}
Some comments on Algorithm~\ref{algo:full} are in order.
\begin{enumerate}
\item\label{r:OnAlgo_explicit}
Algorithm~\ref{algo:full} is generally implicit but becomes explicit as soon as $N$, $P$, $Q$ are assumed to be lower triangular with zeros on the diagonal, $R$ is lower triangular, and $Q_{(j+1)j} =\dots =Q_{ij} =0$ whenever $P_{ij}\neq 0$ with $i>j$. In the latter case, the computation of $x_1^k$ requires only information from some $z_j^k$ with $j\in \{1,\dots, m\}$, while for each $i\in \{2,\dots, n\}$, the component $x_i^k$ depends only on some $z_j^k$ with $j\in\{1,\dots,m\}$ and previously computed components $x_j^k$ with $j<i$. Note that $Q_{1j} = \dots = Q_{jj} = 0$ due to the lower triangular structure of $Q$. Therefore, if $P_{ij}\neq 0$, then setting $Q_{(j+1)j} = \dots = Q_{ij} = 0$ ensures that $B_j(\dots + P_{ij}x_i^k + \dots)$ in $\mathbf{B}(P^{\top}\mathbf{x}^k)$ is excluded from the computation of $x_1^k, \dots, x_i^k$. Specifically, Algorithm~\ref{algo:full} can be written explicitly as
\begin{align*}
\begin{cases}
x^k_1 &=J_{\frac{\gamma}{\delta_1}A_1}\Bigg(\frac{1}{\delta_1}\sum_{j=1}^m M_{1j}z^k_j\Bigg), \\
x^k_i &=J_{\frac{\gamma}{\delta_i}A_i}\Bigg(\frac{1}{\delta_i}\Big(\sum_{j=1}^m M_{ij}z^k_j + \sum_{j=1}^{i-1} N_{ij}x^k_j - \gamma \sum_{j=1}^{\min\{i-1, p\}} (P_{ij} -Q_{ij})B_j\left(\sum_{l=1}^{j} R_{jl}x^k_l\right) \\ 
&\qquad\qquad\qquad - \gamma \sum_{j=1}^{\min\{i-1, p\}} Q_{ij}B_j\left(\sum_{l=j+1}^{n} P_{lj}x^k_l\right)\Big)\Bigg),\quad i \in \{2, \dots, n\}, \\
z^{k+1}_i\!\!\!\!\! &=z^k_i -\lambda_k\sum_{j=1}^n M_{ji}x^k_j,\quad i\in \{1, \dots, m\}.
\end{cases}
\end{align*}
Nevertheless, the subsequent analysis does not require Algorithm~\ref{algo:full} to be explicit and remains valid for the general implicit case. The explicit formulation is thus a practical convenience rather than a theoretical necessity.

\item
When $Q =0$, Algorithm~\ref{algo:full} reduces to
\begin{align*}
\begin{cases}
x^k_i &=J_{\frac{\gamma}{\delta_i}A_i}\left(\frac{1}{\delta_i}\left(\sum_{j=1}^m M_{ij}z^k_j + \sum_{j=1}^n N_{ij}x^k_j - \gamma \sum_{j=1}^p P_{ij}B_j\left(\sum_{l=1}^n R_{jl}x^k_l\right)\right)\right),\\
&\qquad\qquad \qquad i \in \{1, \dots, n\}, \\
z^{k+1}_i \!\!\!\!\!&=z^k_i -\lambda_k\sum_{j=1}^n M_{ji}x^k_j,\quad i\in \{1, \dots, m\},
\end{cases}
\end{align*}
or equivalently,
\begin{align*}
\begin{cases}
\mathbf{x}^k &=J_{\gamma D^{-1}\mathbf{A}}( D^{-1}( M\mathbf{z}^k +  N\mathbf{x}^k -\gamma  P\mathbf{B}(R\mathbf{x}^k))), \\
\mathbf{z}^{k+1} \!\!\!\!\!&=\mathbf{z}^k - \lambda_k M^{\top}\mathbf{x}^k.
\end{cases}
\end{align*}

\item\label{r:OnAlgo_operator}
The sequence $(\mathbf{z}^k)_{k\in \mathbb{N}}$ generated by Algorithm~\ref{algo:full} can be obtained via
\begin{align*}
\mathbf{z}^{k+1} =\left(1 -\frac{\lambda_k}{\theta}\right)\mathbf{z}^k +\frac{\lambda_k}{\theta} T\mathbf{z}^k,    
\end{align*}
where $T\colon \mathcal{H}^{m} \to \mathcal{H}^{m}$ defined by
\begin{align*}
T\mathbf{z} := \mathbf{z} -\theta M^{\top}\mathbf{x}
\end{align*}
with $\theta\in (0, +\infty)$ and $\mathbf{x}=(x_1,\dots, x_n) \in \mathcal{H}^n$ given via the solution operator $S\colon\mathcal{H}^m\to\mathcal{H}^n$
\begin{align*}
\mathbf{x} = S\mathbf{z} := J_{\gamma D^{-1}\mathbf{A}}( D^{-1}( M\mathbf{z} +  N\mathbf{x} -\gamma (P -Q)\mathbf{B}(R\mathbf{x}) -\gamma Q\mathbf{B}(P^{\top}\mathbf{x}))).
\end{align*}

\item\label{r:OnAlgo_xz}
Given the sequences $(\mathbf{z}^k)_{k\in \mathbb{N}}$ and $(\mathbf{x}^k)_{k\in \mathbb{N}}$ generated by Algorithm~\ref{algo:full}, for all $s =(s_1, \dots, s_n)^\top \in \ran M$, we have that $s = Mv$ for some $v\in \mathbb{R}^{m}$, and hence, for all $k\in \mathbb{N}$,
\begin{align*}
\sum_{i=1}^n s_i x_i^k =(s^\top \otimes \Id)\mathbf{x}^k =(v^\top M^{\top} \otimes \Id)\mathbf{x}^k =(v^\top \otimes \Id) M^{\top}\mathbf{x}^k =\frac{1}{\theta}(v^\top \otimes \Id)(\Id -T)\mathbf{z}^k.   
\end{align*}

\item\label{r:OnAlgo_reduced}
In general, $m\geq n-1$, where $m=n-1$ corresponds to the \emph{minimal lifting} defined in \cite{Ryu20}. In the case when $m >n$, by setting $\mathbf{v}^k = M\mathbf{z}^k$, Algorithm~\ref{algo:full} becomes
\begin{align*}
\begin{cases}
\mathbf{x}^k &=J_{\gamma D^{-1}\mathbf{A}}( D^{-1}(\mathbf{v}^k +  N\mathbf{x}^k -\gamma (P-Q)\mathbf{B}(R\mathbf{x}^k)-\gamma Q\mathbf{B}(P^{\top}\mathbf{x}^k))), \\
\mathbf{v}^{k+1} \!\!\!\!\!&=\mathbf{v}^k - \lambda_k M M^{\top}\mathbf{x}^k.
\end{cases}
\end{align*}
This reformulation allows for working in a reduced-dimensional space with $\mathbf{v}^k \in \mathcal{H}^n$, rather than using $\mathbf{z}^k\in \mathcal{H}^m$ from the higher-dimensional space. A similar idea was used in \cite[Remark 2.7]{Tam23}.
\end{enumerate}
\end{remark}

In view of Remark~\ref{r:OnAlgo}\ref{r:OnAlgo_operator}, we start our analysis by investigating the properties of operator $T$. From now on, set $\bOne_n :=(1,\dots,1)\in\mathbb{R}^n$ and omit the subscript when the dimension is clear from the context. For example, if $R\in\mathbb{R}^{p\times n}$, then $R\bOne = \bOne$ is understood with the first $\bOne$ in $\mathbb{R}^n$ and the second in $\mathbb{R}^p$. We consider the following assumptions on the coefficient matrices.

\begin{assumption}
\label{a:stand}
The coefficient matrices $D\in\mathbb{R}^{n\times n}$, $M \in \mathbb{R}^{n \times m}$, $N \in \mathbb{R}^{n \times n}$, $P\in\mathbb{R}^{n\times p}$, and $R \in \mathbb{R}^{p\times n}$ satisfy the following properties.
\begin{enumerate}
 \item\label{a:stand_kerM}
$\ker M^{\top} = \lspan\{\bOne\}$.
\item\label{a:stand_N}
$\sum_{i,j=1}^n N_{ij} = \sum_{i=1}^n \delta_i$.
\item\label{a:stand_P}
$P^{\top}\bOne=\bOne$.
\item\label{a:stand_R}
$R\bOne=\bOne$.
\item\label{a:stand_neg_semidef}
$2D - N - N^{\top} - MM^{\top}  \succeq 0$.
\end{enumerate}
\end{assumption}

In what follows, we denote
\begin{align*}
\Delta := \{ \mathbf{x} = (x_1, \dots, x_n) \in \mathcal{H}^n : x_1 = \dots = x_n \}.    
\end{align*}

\begin{remark}[Implications of the assumption]
\label{r:OnAssumption}
~
\begin{enumerate}
\item\label{r:OnAssumption_kerM} 
Assumption~\ref{a:stand}\ref{a:stand_kerM} implies that
\begin{align*}
\ker(M^{\top}\otimes \Id) = \Delta.    
\end{align*}
Indeed, let $\mathbf{x} = (x_1, \dots, x_n) \in \ker(M^{\top}\otimes \Id)$. By properties of the Kronecker product, for all $v \in \mathcal{H}$, $(\langle x_1, v \rangle, \dots, \langle x_n, v \rangle) \in \ker M^{\top}$, which together with Assumption~\ref{a:stand}\ref{a:stand_kerM} yields $\langle x_1, v \rangle = \dots = \langle x_n, v \rangle$. Therefore, $x_1 = \dots = x_n$, and so $\mathbf{x} \in \Delta$. Conversely, if $\mathbf{x} = (x_1, \dots, x_n)\in \Delta$, then $x_1 = \dots = x_n$, which means $\mathbf{x} \in \ker(M^{\top}\otimes \Id)$.
\item\label{r:OnAssumption_PR}
By Assumption~\ref{a:stand}\ref{a:stand_P}--\ref{a:stand_R}, for all $\mathbf{x} =(x,\dots, x)\in \Delta$, $ P^{\top}\mathbf{x} = R\mathbf{x} =\mathbf{x}$ and
\begin{align*}
(P -Q)\mathbf{B}(R\mathbf{x}) +Q\mathbf{B}(P^{\top}\mathbf{x}) =(P -Q)\mathbf{B}\mathbf{x} +Q\mathbf{B}\mathbf{x} = P\mathbf{B}\mathbf{x},   
\end{align*}
which implies that
\begin{align*}
\sum_{i=1}^n \left((P -Q)\mathbf{B}(R\mathbf{x}) +Q\mathbf{B}(P^{\top}\mathbf{x})\right)_i =\sum_{j=1}^{p} B_j x.    
\end{align*}

\item\label{r:OnAssumption_KU}
In our analysis, we will need to write $P^{\top}-R =UM^{\top}$ and $P^{\top}-Q^{\top} =KM^{\top}$ for some matrices $U$ and $K$. For this purpose, Assumption~\ref{a:stand}\ref{a:stand_kerM},\ref{a:stand_P}--\ref{a:stand_R} implies that $U =(P^{\top} -R)(M^{\top})^\dag \in\mathbb{R}^{p\times m}$ is the minimal norm solution satisfying $UM^{\top}=P^{\top}-R$. Moreover, if $Q^{\top}\bOne =\bOne $, then $K=(P^{\top}-Q^{\top})(M^{\top})^\dag\in \mathbb{R}^{p\times m}$ is the minimal norm solution satisfying $KM^{\top}=P^{\top}-Q^{\top}$. 

To justify the first point, observe that $(P^{\top}-R)\bOne= \bOne - \bOne = 0$ due to Assumption~\ref{a:stand}\ref{a:stand_P}--\ref{a:stand_R}. Therefore, using Lemma~\ref{l:ran_ker}, $\ran(P-R^{\top})\subseteq (\ker M^{\top})^\perp = \ran M$. By combining with Lemma~\ref{l:matrix}, $U =(P^{\top} -R)(M^{\top})^\dag \in\mathbb{R}^{p\times m}$ is the minimal norm solution satisfying $UM^{\top}=P^{\top}-R$. 

Similarly, if $Q^{\top}\bOne =\bOne $, then $(P^{\top}-Q^{\top})\bOne=\bOne -\bOne=0$, which yields $\ran(P-Q) \subseteq (\ker M^{\top})^\perp = \ran M$, and so $K=(P^{\top}-Q^{\top})(M^{\top})^\dag\in \mathbb{R}^{p\times m}$ is the minimal norm solution satisfying $KM^{\top}=P^{\top}-Q^{\top}$.
\end{enumerate}
\end{remark}

\begin{remark}[Simple selections for $P$, $Q$, and $R$]
\label{r:PQR}
As mentioned in Remark~\ref{r:OnAlgo}\ref{r:OnAlgo_explicit}, to make Algorithm~\ref{algo:full} explicit, we require $P$ and $Q$ to be lower triangular with zeros on the diagonal, $R$ to be lower triangular, and $Q_{(j+1)j}=\dots=Q_{ij}=0$ whenever $P_{ij}\neq0$ with $i>j$. For our analysis, we also impose $P^{\top}\bOne=\bOne $, $R\bOne=\bOne $, and either $Q^{\top}\bOne=\bOne $ or $Q=0$. In practice, however, $P$, $Q$, and $R$ are typically chosen so that each column of $P$ and $Q$ (if $Q\neq 0$) and each row of $R$ contains exactly one nonzero entry, which is equal to $1$. Examples of $P\in\mathbb{R}^{n\times p}$, $R\in\mathbb{R}^{p\times n}$, and $Q\in\mathbb{R}^{n\times p}$ include 
\begin{align}
    \label{eq:PQR1}
      P &=\left[ 
        \begin{array}{c} 
          0_{1\times p} \\ 
          \hline 
          \Id_{p} \\
          \hline
          0_{(n-1-p)\times p}
        \end{array} 
        \right],
    &
    Q &=\left[ 
        \begin{array}{c} 
          0_{(n-p)\times p} \\ 
          \hline 
          \Id_{p} 
        \end{array} 
        \right], 
  & 
    R &=\left[ 
        \begin{array}{c|c} 
          \Id_{p} & 0_{p\times (n-p)}
        \end{array} 
        \right], \\
    \label{eq:PQR2}
    P &=\left[ 
        \begin{array}{c} 
          0_{p\times p} \\ 
          \hline 
          1_{1\times p} \\
          \hline
          0_{(n-1-p)\times p}
        \end{array} 
        \right], 
  & 
      Q &=\left[ 
    \begin{array}{c} 
      0_{(n-1) \times (n-2)} \\ 
      \hline 
      1_{1 \times (n-2)} 
    \end{array} 
    \right], 
  & 
    R &=\left[
        \begin{array}{c|c}
            1_{p\times 1} & 0_{p\times (n-1)}               
    \end{array}\right].
\end{align}
It is clear that these selections satisfy Assumption~\ref{a:stand}\ref{a:stand_P}--\ref{a:stand_R} and $Q^{\top}\bOne=\bOne $. 
\end{remark}

\begin{lemma}[Fixed points and zeros]
\label{l:fixzeros}
Suppose that Assumption~\ref{a:stand}\ref{a:stand_kerM}--\ref{a:stand_R} holds. Then the following hold:
    \begin{enumerate}
    \item\label{l:fixzeros_T}
    If $\mathbf{z} \in \Fix T$ and $\mathbf{x} = S\mathbf{z}$, then $\mathbf{x}=(x, \dots, x) \in \Delta$ and $x\in \zer\left( \sum_{i=1}^n A_i + \sum_{j=1}^{p} B_j \right)$.
    \item\label{l:fixzeros_AB}
    If $x \in \zer\left( \sum_{i=1}^n A_i + \sum_{j=1}^{p} B_j \right)$, then there exists $\mathbf{z} \in \Fix T$ such that $\mathbf{x} = S\mathbf{z}$ and $\mathbf{x}=(x, \dots, x) \in \Delta$.
    \end{enumerate}
Consequently, $\Fix T \neq \varnothing$ if and only if $\zer\left( \sum_{i=1}^n A_i + \sum_{j=1}^{p} B_j \right) \neq \varnothing$.
\end{lemma}

\begin{proof}
\ref{l:fixzeros_T}: Assume $\mathbf{z} \in \Fix T$ and $\mathbf{x} = S\mathbf{z}$. Then $\mathbf{z}=T\mathbf{z}=\mathbf{z}-\theta M^{\top}\mathbf{x}$, which yields $ M^{\top}\mathbf{x}=0$, that is, $\mathbf{x} \in \ker(M^{\top}\otimes \Id)$. In view of Remark~\ref{r:OnAssumption}\ref{r:OnAssumption_kerM}, $\mathbf{x} = (x, \dots, x) \in \Delta$. 

Next, it follows from $\mathbf{x} = S\mathbf{z}$ that $ D^{-1}( M \mathbf{z} +  N \mathbf{x}) - \mathbf{x} \in \gamma D^{-1}(\mathbf{A}\mathbf{x} + (P-Q)\mathbf{B}(R\mathbf{x})+Q\mathbf{B}(P^{\top}\mathbf{x}))$, which is equivalent to $ M\mathbf{z}+ N\mathbf{x}- D\mathbf{x} \in \gamma\mathbf{A}\mathbf{x}+\gamma(P-Q)\mathbf{B}(R\mathbf{x})+\gamma Q\mathbf{B}(P^{\top}\mathbf{x})$.
By combining with Remark~\ref{r:OnAssumption}\ref{r:OnAssumption_PR} and Assumption~\ref{a:stand}\ref{a:stand_kerM}--\ref{a:stand_N},
    \begin{align*}
        \gamma \sum_{i=1}^n A_i x + \gamma \sum_{j=1}^{p}B_j x &= \gamma \sum_{i=1}^n(\mathbf{A}\mathbf{x})_i + \gamma\sum_{i=1}^n((P-Q)\mathbf{B}(R\mathbf{x})+Q\mathbf{B}(P^{\top}\mathbf{x}))_i \\
        &\ni ((\bOne^\top M) \otimes \Id) \mathbf{z} + \sum_{i,j=1}^n N_{ij} x - \sum_{i=1}^n \delta_ix \\
        &=((M^{\top}\bOne)^\top \otimes \Id) \mathbf{z} = 0.
    \end{align*}
Thus, $x \in \zer\left(\sum_{i=1}^n A_i + \sum_{j=1}^{p} B_j\right)$.

\ref{l:fixzeros_AB}: Assume $x \in \zer \left( \sum_{i=1}^n A_i + \sum_{j=1}^{p} B_j \right)$ and set $\mathbf{x} := (x, \dots, x) \in \Delta$. Then there exists $\mathbf{v} = (v_1, \dots, v_n) \in \mathbf{A} \mathbf{x}$ such that $\sum_{i=1}^n v_i +\sum_{j=1}^{p} B_j x = 0$. Set $\mathbf{y} = (y_1, \dots, y_n) \in \mathcal{H}^n$ where $\mathbf{y} := \gamma D^{-1}\mathbf{v} + \mathbf{x}$. Then $\mathbf{y}-\mathbf{x}=\gamma D^{-1} \mathbf{v} \in \gamma  D^{-1}\mathbf{A}\mathbf{x}$, and so $\mathbf{x} = J_{\gamma D^{-1}\mathbf{A}} (\mathbf{y})$. 

Now, we show that there exists $\mathbf{z} \in \mathcal{H}^{m}$ such that
\begin{align*}
\mathbf{y} = D^{-1}( M \mathbf{z} +  N \mathbf{x} - \gamma(P-Q)\mathbf{B}(R\mathbf{x})-\gamma Q\mathbf{B}(P^{\top}\mathbf{x})),    
\end{align*}
or equivalently, $ M\mathbf{z} = D\mathbf{y} -  N \mathbf{x} + \gamma(P-Q)\mathbf{B}(R\mathbf{x})+\gamma Q\mathbf{B}(P^{\top}\mathbf{x})$. It suffices to prove $ D\mathbf{y} -  N \mathbf{x} + \gamma(P-Q)\mathbf{B}(R\mathbf{x})+\gamma Q\mathbf{B}(P^{\top}\mathbf{x}) \in \ran(M\otimes \Id) = (\ker(M\otimes \Id)^{\top})^\perp = \Delta^\perp = \left\{ (x_1, \dots, x_n) \in \mathcal{H}^n : \sum_{i=1}^n x_i = 0 \right\}$. Indeed, using Remark~\ref{r:OnAssumption}\ref{r:OnAssumption_PR} and Assumption~\ref{a:stand}\ref{a:stand_kerM}--\ref{a:stand_N}, we have that
\begin{align*}
&\sum_{i=1}^n ( D\mathbf{y} -  N \mathbf{x} + \gamma (P-Q)\mathbf{B}(R\mathbf{x})+\gamma Q\mathbf{B}(P^{\top}\mathbf{x}))_i \notag \\
&= \sum_{i=1}^n \delta_i y_i - \sum_{i,j=1}^n N_{ij} x + \gamma\sum_{j=1}^{p}B_j x \\
&= \gamma\left( \sum_{i=1}^n v_i + \sum_{j=1}^{p}B_j x \right) + \sum_{i=1}^n \delta_i x_i - \sum_{i,j=1}^n N_{ij} x = 0.
\end{align*}
Therefore, $\mathbf{x} = S\mathbf{z}$ and $T\mathbf{z} =\mathbf{z} -\theta M^{\top}\mathbf{x}$. Since $\mathbf{x} \in \Delta$, it follows from Remark~\ref{r:OnAssumption}\ref{r:OnAssumption_kerM} that $M^{\top}\mathbf{x} =0$, hence $\mathbf{z} \in \Fix T$.
\end{proof}

\begin{lemma}[Conical (quasi)averagedness]
\label{l:nonexpansive}
Suppose that $A_1, \dots, A_n$ are maximally monotone and that Assumption~\ref{a:stand}\ref{a:stand_kerM}--\ref{a:stand_R} hold.
\begin{enumerate}
\item\label{Lipschitz}
Suppose that $B_1, \dots, B_{p}$ are monotone and ${\ell}$-Lipschitz continuous, and that $Q^{\top}\bOne =\bOne $. Then, for all $\mathbf{z}\in \mathcal{H}^{m}$, $\bar{\mathbf{z}}\in \Fix T$,
\begin{multline}\label{quasinonexpansive}
\| T\mathbf{z}-\bar{\mathbf{z}}\|^2 + \left(\frac{1-\theta}{\theta}-\frac{\gamma \ell}{\theta}\left(\|K\|^2+\|U\|^2\right)\right)\|(\Id-T)\mathbf{z}\|^2 \\
+ \theta\langle \mathbf{x}-\bar{\mathbf{x}}, [2 D-  N -  N^{\top}- M M^{\top} ]\mathbf{x}\rangle \leq \|\mathbf{z}-\bar{\mathbf{z}}\|^2,
\end{multline}
where $\mathbf{x} = S\mathbf{z}$, $\bar{\mathbf{x}} = S\bar{\mathbf{z}}$ and $K=(P^{\top}-Q^{\top})(M^{\top})^\dag$, $U=(P^{\top}-R)(M^{\top})^\dag$. In particular, if Assumption~\ref{a:stand}\ref{a:stand_neg_semidef} holds and $$\gamma \in \left(0, \frac{1}{\ell\left(\|K\|^2+\|U\|^2\right)}\right),$$ then $T$ is conically $\rho$-quasiaveraged with $\rho=\frac{\theta}{1-\gamma\ell\left(\|K\|^2+\|U\|^2\right)} > 0$.
    
\item\label{cocoercive}
Suppose that $B_1, \dots, B_p$ are $\frac{1}{\ell}$-cocoercive and that $Q=0$. Then, for all $\mathbf{z}, \bar{\mathbf{z}} \in \mathcal{H}^{m}$,
\begin{multline}\label{averagedness}
\| T\mathbf{z} - T\bar{\mathbf{z}} \|^2 + \left(\frac{1 - \theta}{\theta} -\frac{\gamma \ell}{2\theta}\|U\|^2\right)\| (\Id - T)\mathbf{z} - (\Id - T)\bar{\mathbf{z}} \|^2 \\
+\frac{\gamma\ell}{2\theta}\left\| U(\Id-T)\mathbf{z} - U(\Id-T)\bar{\mathbf{z}} +\frac{2\theta}{\ell}(\mathbf{B}(R\mathbf{x}) -\mathbf{B}(R\bar{\mathbf{x}}))\right\|^2 \\
+\theta \langle \mathbf{x} - \bar{\mathbf{x}}, [2 D-  N -  N^{\top} -  M  M^{\top} ] (\mathbf{x} - \bar{\mathbf{x}}) \rangle \leq \| \mathbf{z} - \bar{\mathbf{z}} \|^2,
\end{multline}
where $\mathbf{x} = S\mathbf{z}$, $\bar{\mathbf{x}} = S\bar{\mathbf{z}}$, and $U =(P^{\top} -R)(M^{\top})^\dag$. In particular, if Assumption~\ref{a:stand}\ref{a:stand_neg_semidef} holds and $$\gamma \in \left( 0, \frac{2}{ \ell\|U\|^2} \right),$$ then $T$ is conically $\rho$-averaged with $\rho = \frac{2\theta}{2-\gamma \ell\|U\|^2} > 0$.
\end{enumerate}
Here, we adopt the convention that $\frac{1}{\ell\left(\|K\|^2+\|U\|^2\right)}=+\infty$ if $\|K\|^2+\|U\|^2 =0$, and $\frac{2}{\ell\|U\|^2}=+\infty$ if $\|U\|=0$.
\end{lemma}
\begin{proof}
Let $\mathbf{z}, \bar{\mathbf{z}} \in \mathcal{H}^{m}$ and set $\mathbf{y} :=  D^{-1}( M \mathbf{z} +  N \mathbf{x} - \gamma (P-Q)\mathbf{B}(R\mathbf{x})-\gamma Q\mathbf{B}(P^{\top}\mathbf{x}))$, $\bar{\mathbf{y}} :=  D^{-1}( M \bar{\mathbf{z}} +  N \bar{\mathbf{x}} - \gamma(P-Q)\mathbf{B}(R\bar{\mathbf{x}})-\gamma Q\mathbf{B}(P^{\top}\bar{\mathbf{x}}))$. Then $\mathbf{x} = J_{\gamma  D^{-1}\mathbf{A}}(\mathbf{y})$ and $\bar{\mathbf{x}} = J_{\gamma D^{-1}\mathbf{A}}(\bar{\mathbf{y}})$, or equivalently, $ D\mathbf{y}- D\mathbf{x} \in \gamma\mathbf{A} \mathbf{x}$ and $ D\bar{\mathbf{y}} -  D\bar{\mathbf{x}} \in \gamma\mathbf{A}\bar{\mathbf{x}}$. The monotonicity of $\mathbf{A} = (A_1, \dots, A_n)$ gives
\begin{align}
0 &\leq \langle \mathbf{x} - \bar{\mathbf{x}},  D(\mathbf{y}-\mathbf{x}) -  D(\bar{\mathbf{y}} - \bar{\mathbf{x}}) \rangle \notag \\ 
&= \langle \mathbf{x} - \bar{\mathbf{x}}, ( M\mathbf{z} +  N\mathbf{x} - \gamma(P-Q)\mathbf{B}(R\mathbf{x}) - \gamma Q\mathbf{B}(P^{\top}\mathbf{x}) -  D\mathbf{x}) \notag \\
&\quad- ( M\bar{\mathbf{z}} +  N\bar{\mathbf{x}}-\gamma(P-Q)\mathbf{B}(R\bar{\mathbf{x}}) -\gamma Q\mathbf{B}(P^{\top}\bar{\mathbf{x}}) -  D\bar{\mathbf{x}}) \rangle  \notag \\
&= \langle  M^{\top}\mathbf{x} -  M^{\top}\bar{\mathbf{x}}, \mathbf{z} - \bar{\mathbf{z}} \rangle + \langle \mathbf{x} - \bar{\mathbf{x}}, ( N -  D)\mathbf{x} - ( N- D)\bar{\mathbf{x}} \rangle \notag \\
&\quad - \gamma \langle \mathbf{x} - \bar{\mathbf{x}}, (P-Q)\mathbf{B}(R\mathbf{x}) + Q\mathbf{B}(P^{\top}\mathbf{x}) - (P-Q)\mathbf{B}(R\bar{\mathbf{x}}) - Q\mathbf{B}(P^{\top}\bar{\mathbf{x}}) \rangle. \label{monotonicity}   
\end{align}
The first term in \eqref{monotonicity} is expressed as
\begin{align}
& \langle  M^{\top}\mathbf{x}- M^{\top}\bar{\mathbf{x}}, \mathbf{z}-\bar{\mathbf{z}} \rangle \notag \\
&= \frac{1}{\theta} \langle (\Id-T)\mathbf{z}-(\Id - T)\bar{\mathbf{z}}, \mathbf{z} - \bar{\mathbf{z}} \rangle \notag \\
&= \frac{1}{2\theta} (\|\mathbf{z} - \bar{\mathbf{z}}\|^2 + \|(\Id-T)\mathbf{z}-(\Id-T)\bar{\mathbf{z}}\|^2 - \|T\mathbf{z}-T\bar{\mathbf{z}}\|^2),
\label{cocoercive_firstterm}
\end{align}
while the second term in \eqref{monotonicity} is written as
\begin{align}
&\langle \mathbf{x} - \bar{\mathbf{x}}, ( N -  D)\mathbf{x} - ( N- D)\bar{\mathbf{x}} \rangle \notag \\
&= \frac{1}{2}\langle \mathbf{x}-\bar{\mathbf{x}}, [ M M^{\top} +2 N - 2 D](\mathbf{x}-\bar{\mathbf{x}})\rangle - \frac{1}{2}\| M^{\top}\mathbf{x} -  M^{\top}\bar{\mathbf{x}}\|^2 \notag \\
&= \frac{1}{2} \langle \mathbf{x}-\bar{\mathbf{x}}, [ M M^{\top}+ N +  N^{\top} - 2 D](\mathbf{x}-\bar{\mathbf{x}})\rangle - \frac{1}{2\theta^2}\|(\Id-T)\mathbf{z}-(\Id-T)\bar{\mathbf{z}}\|^2.
\label{cocoercive_secondterm}
\end{align}

\ref{Lipschitz}: As $\bar{\mathbf{z}}\in \Fix T$, we have from Lemma~\ref{l:fixzeros}\ref{l:fixzeros_T} that $\bar{\mathbf{x}} = (\bar{x}, \dots, \bar{x})\in \Delta$. Then \eqref{cocoercive_firstterm} becomes 
\begin{align}
\label{lipschitz_firstterm}
\frac{1}{2\theta}(\|\mathbf{z}-\bar{\mathbf{z}}\|^2 + \|(\Id-T)\mathbf{z}\|^2 - \|T\mathbf{z}-\bar{\mathbf{z}}\|^2).
\end{align}
Using Assumption~\ref{a:stand}\ref{a:stand_kerM}--\ref{a:stand_N}, \eqref{cocoercive_secondterm} becomes
\begin{align}
\label{lipschitz_secondterm}
&\frac{1}{2} \langle \mathbf{x}-\bar{\mathbf{x}}, [ M M^{\top} +  N +  N^{\top}-2 D]\mathbf{x}\rangle - \frac{1}{2\theta^2}\|(\Id-T)\mathbf{z}\|^2.
\end{align}
The last term in \eqref{monotonicity} can be written as 
\begin{align*}
&-\gamma\langle \mathbf{x}-\bar{\mathbf{x}},  P(\mathbf{B}(R\mathbf{x})-\mathbf{B}(R\bar{\mathbf{x}}))\rangle- \gamma\langle \mathbf{x}-\bar{\mathbf{x}}, Q[(\mathbf{B}(P^{\top}\mathbf{x})-\mathbf{B}(P^{\top}\bar{\mathbf{x}}))-(\mathbf{B}(R\mathbf{x})-\mathbf{B}(R\bar{\mathbf{x}}))]\rangle.
\end{align*}
Using the monotonicity of $B_1, \dots, B_{p}$ and Remark~\ref{r:OnAssumption}\ref{r:OnAssumption_PR}, it follows that
        \begin{align*}
        0&\leq \gamma\langle  P^{\top}(\mathbf{x}-\bar{\mathbf{x}}), (\mathbf{B}(P^{\top}\mathbf{x})-\mathbf{B}(P^{\top}\bar{\mathbf{x}}))\rangle \\
        &= \gamma\langle  P^{\top}(\mathbf{x}-\bar{\mathbf{x}}), (\mathbf{B}(P^{\top}\mathbf{x})-\mathbf{B}(R\mathbf{x}))\rangle + \gamma\langle  P^{\top}(\mathbf{x}-\bar{\mathbf{x}}), \mathbf{B}(R\mathbf{x})-\mathbf{B}(P^{\top}\bar{\mathbf{x}})\rangle \\
        &= \gamma\langle  P^{\top}(\mathbf{x}-\bar{\mathbf{x}}), (\mathbf{B}(P^{\top}\mathbf{x})-\mathbf{B}(R\mathbf{x}))\rangle + \gamma\langle  P^{\top}(\mathbf{x}-\bar{\mathbf{x}}), \mathbf{B}(R\mathbf{x})-\mathbf{B}(R\bar{\mathbf{x}})\rangle
            \end{align*}
    which implies
        \begin{align}
            -\gamma\langle \mathbf{x}-\bar{\mathbf{x}},  P(\mathbf{B}(R\mathbf{x})-\mathbf{B}(R\bar{\mathbf{x}}))\rangle 
            &\leq \gamma\langle  P^{\top}(\mathbf{x}-\bar{\mathbf{x}}), (\mathbf{B}(P^{\top}\mathbf{x})-\mathbf{B}(R\mathbf{x}))\rangle.
        \label{lipschitz_lastterm1}
        \end{align}
    With the help of  Remark~\ref{r:OnAssumption}\ref{r:OnAssumption_PR} and the assumption that $Q^{\top}\bOne=\bOne $, we have 
        \begin{align}
            -\,& \gamma\langle \mathbf{x}-\bar{\mathbf{x}}, Q[(\mathbf{B}(P^{\top}\mathbf{x})-\mathbf{B}(P^{\top}\bar{\mathbf{x}}))-(\mathbf{B}(R\mathbf{x})-\mathbf{B}(R\bar{\mathbf{x}}))]\rangle \notag \\
            &= -\gamma\langle Q^{\top}(\mathbf{x}-\bar{\mathbf{x}}), (\mathbf{B}(P^{\top}\mathbf{x})-\mathbf{B}(P^{\top}\bar{\mathbf{x}}))-(\mathbf{B}(R\mathbf{x})-\mathbf{B}(R\bar{\mathbf{x}}))\rangle \notag \\
            &= -\gamma\langle Q^{\top}(\mathbf{x}-\bar{\mathbf{x}}), \mathbf{B}(P^{\top}\mathbf{x})-\mathbf{B}(R\mathbf{x})\rangle \notag \\
            &= -\gamma\langle Q^{\top}\mathbf{x}- P^{\top}\mathbf{x}, \mathbf{B}(P^{\top}\mathbf{x})-\mathbf{B}(R\mathbf{x})\rangle - \gamma\langle  P^{\top}\mathbf{x}-Q^{\top}\bar{\mathbf{x}}, \mathbf{B}(P^{\top}\mathbf{x})-\mathbf{B}(R\mathbf{x})\rangle  \notag \\
            &= \gamma\langle (P^{\top}-Q^{\top})\mathbf{x}, \mathbf{B}(P^{\top}\mathbf{x})-\mathbf{B}(R\mathbf{x})\rangle - \gamma\langle  P^{\top}(\mathbf{x}-\bar{\mathbf{x}}), \mathbf{B}(P^{\top}\mathbf{x})-\mathbf{B}(R\mathbf{x})\rangle.
        \label{lipschitz_lastterm2}
        \end{align}
    Using Cauchy--Schwarz inequality, $\ell$-Lipschitz continuity of $B_1, \dots, B_{p}$ and Remark~\ref{r:OnAssumption_KU}, we have the estimation 
        \begin{align}
            \gamma\langle (P^{\top}-Q^{\top})\mathbf{x}, \mathbf{B}(P^{\top}\mathbf{x})-\mathbf{B}(R\mathbf{x})\rangle 
            &\leq \frac{\gamma\ell}{2}(\|(P^{\top}-Q^{\top})\mathbf{x}\|^2 + \|(P^{\top}- R)\mathbf{x}\|^2) \notag \\
            &\leq \frac{\gamma\ell}{2}\left(\| K\|^2+\| U\|^2\right)\| M^{\top}\mathbf{x}\|^2 \notag \\
            &= \frac{\gamma\ell}{2\theta^2}\left(\|K\|^2+\|U\|^2\right)\|(\Id-T)\mathbf{z}\|^2.
        \label{lipschitz_lastterm}
        \end{align} 
    Multiplying (\ref{lipschitz_firstterm}), (\ref{lipschitz_secondterm}), (\ref{lipschitz_lastterm1}), (\ref{lipschitz_lastterm2}) by $2\theta$ gives (\ref{quasinonexpansive}), and note that the inner product on the LHS of \eqref{quasinonexpansive} is non-negative followed by Assumption~\ref{a:stand}\ref{a:stand_neg_semidef}. Altogether, $T$ is conically $\rho$-quasiaveraged follows from the observation that $\frac{1-\rho}{\rho}=\frac{1-\theta}{\theta}-\frac{\gamma\ell}{\theta}(\|K\|^2+\|U\|^2)$ as claimed.

\ref{cocoercive}: Since $Q=0$, using the cocoercivity of $B_1,\dots, B_{p}$ and Remark~\ref{r:OnAssumption_KU}, the last term in \eqref{monotonicity} can be estimated as
\begin{align}
-\,&\gamma \langle \mathbf{x} - \bar{\mathbf{x}},  P\mathbf{B}(R\mathbf{x}) -  P\mathbf{B}(R\bar{\mathbf{x}}) \rangle  \notag \\
&= -\gamma \langle  P^{\top}(\mathbf{x}-\bar{\mathbf{x}}), \mathbf{B}(R\mathbf{x})-\mathbf{B}(R\bar{\mathbf{x}})\rangle \notag \\
&= -\gamma\langle (P^{\top}- R)(\mathbf{x} -\bar{\mathbf{x}}), \mathbf{B}(R\mathbf{x})-\mathbf{B}(R\bar{\mathbf{x}})\rangle -\gamma\langle  R\mathbf{x} - R\bar{\mathbf{x}}, \mathbf{B}(R\mathbf{x})-\mathbf{B}(R\bar{\mathbf{x}})\rangle) \notag \\
&\leq \frac{\gamma\ell}{4}\|(P^{\top}- R)(\mathbf{x}-\bar{\mathbf{x}})\|^2 -\frac{\gamma\ell}{4}\left\|(P^{\top} - R)(\mathbf{x} -\bar{\mathbf{x}}) +\frac{2}{\ell}(\mathbf{B}(R\mathbf{x}) -\mathbf{B}(R\bar{\mathbf{x}}))\right\|^2 \notag \\ 
&\qquad +\frac{\gamma}{\ell}\|\mathbf{B}(R\mathbf{x}) - \mathbf{B}(R\bar{\mathbf{x}})\|^2 - \frac{\gamma}{\ell}\| \mathbf{B}(R\mathbf{x}) - \mathbf{B}(R\bar{\mathbf{x}})\|^2 \notag \\
&= \frac{\gamma\ell}{4}\| U M^{\top}(\mathbf{x}-\bar{\mathbf{x}})\|^2 -\frac{\gamma\ell}{4}\left\| U M^{\top}(\mathbf{x} -\bar{\mathbf{x}}) +\frac{2}{\ell}(\mathbf{B}(R\mathbf{x}) -\mathbf{B}(R\bar{\mathbf{x}}))\right\|^2 \notag \\ 
&\leq \frac{\gamma\ell}{4} \| U\|^2\| M^{\top}(\mathbf{x}-\bar{\mathbf{x}})\|^2 -\frac{\gamma\ell}{4}\left\| U M^{\top}(\mathbf{x} -\bar{\mathbf{x}}) +\frac{2}{\ell}(\mathbf{B}(R\mathbf{x}) -\mathbf{B}(R\bar{\mathbf{x}}))\right\|^2 \notag \\
&= \frac{\gamma \ell}{4\theta^2} \|U\|^2\|(\Id-T)\mathbf{z}-(\Id-T)\bar{\mathbf{z}}\|^2 \notag \\
&\qquad -\frac{\gamma\ell}{4\theta^2}\left\| U(\Id-T)\mathbf{z} - U(\Id-T)\bar{\mathbf{z}} +\frac{2\theta}{\ell}(\mathbf{B}(R\mathbf{x}) -\mathbf{B}(R\bar{\mathbf{x}}))\right\|^2.
\label{cocoercive_lastterm}
\end{align}
The inequality \eqref{averagedness} follows from observation when multiplying by $2\theta$ and substituting \eqref{cocoercive_firstterm}, \eqref{cocoercive_secondterm}, \eqref{cocoercive_lastterm} into \eqref{monotonicity}. Note that Assumption~\ref{a:stand}\ref{a:stand_neg_semidef} implies that the inner product on the LHS of \eqref{averagedness} is non-negative. The last conclusion follows from the observation that $\frac{1-\rho}{\rho} = \frac{1-\theta}{\theta} - \frac{\gamma \ell}{2\theta}\|U\|^2$. 
\end{proof}

\begin{remark}[A variant of the key inequalities]
\label{r:variant}
    In the proof of Lemma~\ref{l:nonexpansive}, combining the two quadratic forms in $\mathbf{x}$ from the last lines of \eqref{monotonicity} and \eqref{cocoercive_secondterm} leads to a variant of \eqref{quasinonexpansive} and \eqref{averagedness}. 
    \begin{enumerate}
    \item 
    Using \eqref{lipschitz_lastterm2}, the identities $ R\bar{\mathbf{x}}= P^{\top}\bar{\mathbf{x}}=Q^{\top}\bar{\mathbf{x}} = \bar{\mathbf{x}}$, together with Cauchy--Schwarz inequality, and the $\ell$-Lipschitz continuity of $B_1, \dots, B_{p}$, we obtain the estimation
    \begin{align*}
        -\,& \gamma \langle \mathbf{x} - \bar{\mathbf{x}}, (P-Q)\mathbf{B}(R\mathbf{x}) + Q\mathbf{B}(P^{\top}\mathbf{x}) - (P-Q)\mathbf{B}(R\bar{\mathbf{x}}) - Q\mathbf{B}(P^{\top}\bar{\mathbf{x}}) \rangle \notag \\
        &\leq \gamma\langle (P^{\top}-Q^{\top})(\mathbf{x}-\bar{\mathbf{x}}), \mathbf{B}(P^{\top}\mathbf{x})-\mathbf{B}(R\mathbf{x})\rangle. \notag \\
        & \leq \frac{\gamma\ell}{2}(\|(P^{\top}-Q^{\top})(\mathbf{x}-\bar{\mathbf{x}})\|^2 + \|(P^{\top}- R)(\mathbf{x}-\bar{\mathbf{x}})\|^2) \notag \\
        &= \frac{\gamma\ell}{2}\langle \mathbf{x}-\bar{\mathbf{x}}, [(P-Q)(P^{\top}-Q^{\top}) + (P- R^{\top})(P^{\top}- R)](\mathbf{x}-\bar{\mathbf{x}})\rangle. 
    \end{align*}
    The inequality \eqref{quasinonexpansive} then transforms into
    \begin{align}
    \label{new3.2a}
    \| T\mathbf{z}-\bar{\mathbf{z}}\|^2 + \frac{1-\theta}{\theta}\|(\Id-T)\mathbf{z}\|^2
    + \theta\langle \mathbf{x}-\bar{\mathbf{x}},L\mathbf{x}\rangle \leq \|\mathbf{z}-\bar{\mathbf{z}}\|^2,
    \end{align}
    where $L :=2 D-  N -  N^{\top}- M M^{\top} - \gamma\ell((P-Q)(P^{\top}-Q^{\top}) + (P- R^{\top})(P^{\top}- R))$. We deduce that $T$ is conically $\theta$-quasiaveraged if
    \begin{align}\label{eq:semidef_L}
    L =2 D-  N -  N^{\top}- M M^{\top} - \gamma\ell((P-Q)(P^{\top}-Q^{\top}) + (P- R^{\top})(P^{\top}- R)) \succeq 0.
    \end{align}
    \item 
    Since $Q=0$, using the cocoercivity of $B_1,\dots,B_p$, the last term in \eqref{monotonicity} can be estimated as
    \begin{align*}
        -\,&\gamma \langle \mathbf{x} - \bar{\mathbf{x}},  P\mathbf{B}(R\mathbf{x}) -  P\mathbf{B}(R\bar{\mathbf{x}}) \rangle  \notag \\
        &\leq \frac{\gamma\ell}{4}\|(P^{\top}- R)(\mathbf{x}-\bar{\mathbf{x}})\|^2 -\frac{\gamma\ell}{4}\left\|(P^{\top} - R)(\mathbf{x} -\bar{\mathbf{x}}) +\frac{2}{\ell}(\mathbf{B}(R\mathbf{x}) -\mathbf{B}(R\bar{\mathbf{x}}))\right\|^2 \notag \\
        &= \frac{\gamma\ell}{4}\langle \mathbf{x}-\bar{\mathbf{x}}, (P- R^{\top})(P^{\top}- R)(\mathbf{x}-\bar{\mathbf{x}})\rangle \notag \\
        &\quad - \frac{\gamma\ell}{4}\left\|(P^{\top} - R)(\mathbf{x} -\bar{\mathbf{x}}) +\frac{2}{\ell}(\mathbf{B}(R\mathbf{x}) -\mathbf{B}(R\bar{\mathbf{x}}))\right\|^2.
    \end{align*}
    We obtain a variant of \eqref{averagedness} as
    \begin{align}
    \label{new3.2b}
    &\| T\mathbf{z} - T\bar{\mathbf{z}} \|^2 + \frac{1 - \theta}{\theta} \| (\Id - T)\mathbf{z} - (\Id - T)\bar{\mathbf{z}} \|^2 \notag \\
    &+\frac{\theta\gamma\ell}{2}\left\|(P^{\top} - R)(\mathbf{x} -\bar{\mathbf{x}}) +\frac{2}{\ell}(\mathbf{B}(R\mathbf{x}) -\mathbf{B}(R\bar{\mathbf{x}}))\right\|^2 \notag \\
    &+\theta \langle \mathbf{x} - \bar{\mathbf{x}}, [2 D-  N -  N^{\top} -  M  M^{\top} -\frac{\gamma\ell}{2}(P- R^{\top})(P^{\top}- R)] (\mathbf{x} - \bar{\mathbf{x}}) \rangle \leq \| \mathbf{z} - \bar{\mathbf{z}} \|^2. 
    \end{align}
    Therefore, $T$ is conically $\theta$-averaged if
    \begin{align}\label{eq:semidef_C}
    2 D-  N -  N^{\top} -  M  M^{\top} -\frac{\gamma\ell}{2}(P- R^{\top})(P^{\top}- R) \succeq 0.
    \end{align}
    The latter condition coincides with \cite[Assumption~4.7]{ACGN25} when $\gamma=1$ and the cocoercive constants of $B_1,\dots,B_p$ are all equal to $\frac{1}{\ell}$.
    \end{enumerate}
    
    The aforementioned approach improves the conically quasiaveraged (resp., averaged) constant of $T$ to $\theta$ and removes the explicit upper bound on $\gamma$, at the cost of requiring the positive semidefinite condition \eqref{eq:semidef_L} (resp., \eqref{eq:semidef_C}), which involves $\gamma$ and is more restrictive than Assumption~\ref{a:stand}\ref{a:stand_neg_semidef}. As we will see in Section~\ref{s:graphs}, Assumption~\ref{a:stand}\ref{a:stand_neg_semidef} is automatically satisfied while \eqref{eq:semidef_L} and \eqref{eq:semidef_C} might not be. This is especially evident in Remark~\ref{r:smallcases}, the \emph{weighted sequential forward-Douglas--Rachford algorithm} in Example~\ref{eg:ring_seq}, Example~\ref{eg:star}, Remark~\ref{r:ps}, and \emph{Case~1} of Example~\ref{eg:complete}, provided that the underlying graphs have the same weights. In all these cases, $2D-N-N^{\top}-MM^{\top} = 0$ yet $P-R^{\top} \neq 0$, so \eqref{eq:semidef_L} and \eqref{eq:semidef_C} fail, whereas Assumption~\ref{a:stand}\ref{a:stand_neg_semidef} holds.
\end{remark}

\begin{lemma}
\label{l:shadow}
Suppose that $A_1$, \dots, $A_n$ are maximally monotone, that $B_1$, \dots, $B_p$ are monotone and $\ell$-Lipschitz continuous, that Assumption~\ref{a:stand}\ref{a:stand_kerM}--\ref{a:stand_R} holds, and that the first rows of $N$, $P$, and $Q$ are zero. Let $(\mathbf{z}^k)_{k\in \mathbb{N}}$ and $(\mathbf{x}^k)_{k\in \mathbb{N}}$ be the sequences generated by Algorithm~\ref{algo:full}. Suppose further that $(\mathbf{z}^k)_{k\in \mathbb{N}}$ is Fej\'er monotone with respect to $\Fix T$ and that $(\Id -T)\mathbf{z}^k\to 0$ as $k\to +\infty$. Then, as $k\to +\infty$, the following hold:
\begin{enumerate}
\item\label{l:shadow_scvg}
$\sum_{i=1}^n s_ix^k_i\to 0$ whenever $(s_1, \dots, s_n)\in \mathbb{R}^n$ with $\sum_{i=1}^n s_i =0$.
\item\label{l:shadow_wcvg}
$\mathbf{z}^k\rightharpoonup \bar{\mathbf{z}}\in \Fix T$ and $\mathbf{x}^k \rightharpoonup (\bar{x}, \dots, \bar{x}) \in \mathcal{H}^n$ with $\bar{x} \in \zer(\sum_{i=1}^n A_i + \sum_{j=1}^{p} B_j)$.
\end{enumerate}
\end{lemma}
\begin{proof}
\ref{l:shadow_scvg}: By Lemma~\ref{l:ran_ker} and Assumption~\ref{a:stand}\ref{a:stand_kerM},
\begin{align*}
\ran M =(\ker M^{\top})^\perp =\left\{(s_1, \dots, s_n)^\top\in \mathbb{R}^n: \sum_{i=1}^n s_i = 0\right\}.
\end{align*}
Since $(\Id -T)\mathbf{z}^k\to 0$ as $k\to +\infty$, the conclusion then follows from Remark~\ref{r:OnAlgo}\ref{r:OnAlgo_xz}.    

\ref{l:shadow_wcvg}: First, we claim that the sequence $(\mathbf{x}^k)_{k\in \mathbb{N}}$ is bounded. As the first rows of $N$, $P$, and $Q$ are zero, we have $x^k_1 =J_{\frac{\gamma}{\delta_1}A_1}\Big(\frac{1}{\delta_1}\sum_{j=1}^m M_{1j}z^k_j\Big)$ and $x^0_1 =J_{\frac{\gamma}{\delta_1}A_1}\Big(\frac{1}{\delta_1}\sum_{j=1}^m M_{1j}z^0_j\Big)$. By the nonexpansiveness of $J_{\frac{\gamma}{\delta_1}A_1}$, $\|x^k_1 -x^0_1\|\leq \|\frac{1}{\delta_1}\sum_{j=1}^m M_{1j}z^k_j -\frac{1}{\delta_1}\sum_{j=1}^m M_{1j}z^0_j\|$. Since $(\mathbf{z}^k)_{k\in \mathbb{N}}$ is Fej\'er monotone, by \cite[Proposition~5.4(i)]{BC17}, for each $j\in \{1, \dots, n\}$, $(z^k_j)_{k\in \mathbb{N}}$ is bounded, so is $(x_1^k)_{k\in \mathbb{N}}$. Now, by \ref{l:shadow_scvg}, for each $i\in \{2, \dots, n\}$, since $x^k_1 -x^k_i\to 0$ as $k\to +\infty$, the sequence $(x^k_i)_{k\in \mathbb{N}}$ is also bounded. We obtain the boundedness of $(\mathbf{x}^k)_{k\in \mathbb{N}}$.

For each $i\in \{1, \dots, n\}$, we define for $\mathbf{x} =(x_1, \dots, x_n)\in \mathcal{H}^n$,
\begin{align*}
\mathbf{F}_i\mathbf{x} &:=\left((P-Q)\mathbf{B}(R\mathbf{x})+Q\mathbf{B}(P^{\top}\mathbf{x})\right)_i \\
&\ =\sum_{j=1}^{p}\left((P_{ij}-Q_{ij})B_j\left(\sum_{l=1}^n R_{jl}{x_l}\right)+Q_{ij}B_j\left(\sum_{l=1}^n P_{lj}x_l\right)\right)
\end{align*}
and for $x\in \mathcal{H}$, $F_i x :=\mathbf{F}_i \mathbf{x}$ with $\mathbf{x} =(x, \dots, x)\in \Delta$. In view of Remark~\ref{r:OnAssumption}\ref{r:OnAssumption_PR},
\begin{align*}
F_i x =\left(P\mathbf{B}\mathbf{x}\right)_i =\sum_{j=1}^p P_{ij}B_j x.    
\end{align*}
For all $k\in \mathbb{N}$, let $\mathbf{y}^k := D^{-1}( M\mathbf{z}^k +  N \mathbf{x}^k)$. On the one hand, $\mathbf{x}^k = J_{\gamma D^{-1} \mathbf{A}}(\mathbf{y}^k - \gamma D^{-1}((P-Q)\mathbf{B}(R\mathbf{x}^k)-\gamma Q\mathbf{B}(P^{\top}\mathbf{x}^k)))$, and so $\frac{1}{\gamma} D(\mathbf{y}^k -\mathbf{x}^k) -(P-Q)\mathbf{B}(R\mathbf{x}^k)-Q\mathbf{B}(P^{\top}\mathbf{x}^k)\in \mathbf{A}\mathbf{x}^k$, which means that, for all $i\in \{1, \dots, n\}$, $\frac{\delta_i}{\gamma}(y_i^k -x_i^k)-\mathbf{F}_i\mathbf{x}^k\in A_i x_i^k$. Therefore,
\begin{align}\label{eq:inclu}
\begin{bmatrix}
x_1^k -x_n^k \\
x_2^k -x_n^k \\
\vdots \\
x_{n-1}^k -x_n^k \\
\sum_{i=1}^n \frac{\delta_i}{\gamma}(y_i^k -x_i^k) - \sum_{i=1}^n (\mathbf{F}_i\mathbf{x}^k -F_i x_i^k)
\end{bmatrix}
\in
\mathcal{S}\left(\begin{bmatrix}
\frac{\delta_1}{\gamma}(y_1^k -x_1^k)-\mathbf{F}_1\mathbf{x}^k + F_1 x_1^k \\
\frac{\delta_2}{\gamma}(y_2^k -x_2^k)-\mathbf{F}_2\mathbf{x}^k + F_2 x_2^k \\
\vdots \\
\frac{\delta_{n-1}}{\gamma}(y_{n-1}^k -x_{n-1}^k)- \mathbf{F}_{n-1}\mathbf{x}^k +F_{n-1} x_{n-1}^k \\
x_n^k
\end{bmatrix}\right) 
\end{align}
where $\mathcal{S}\colon \mathcal{H}^n \rightrightarrows \mathcal{H}^n$ is the operator given by
\begin{align*}
\mathcal{S} := \diag\left(\begin{bmatrix}
(A_1 + F_1)^{-1} \\
(A_2 + F_2)^{-1} \\
\vdots \\
(A_{n-1} + F_{n-1})^{-1} \\
(A_n + F_n)
\end{bmatrix} \right)
+\begin{bmatrix}
0 & 0 & \dots & 0 & -\Id \\
0 & 0 & \dots & 0 & -\Id \\
\vdots & \ddots & \vdots & \vdots & \vdots\\
0 & 0 & \dots & 0 & -\Id \\
\Id & \Id & \dots & \Id & 0
\end{bmatrix}.    
\end{align*}
Since $B_1, \dots, B_p\colon \mathcal{H}\to \mathcal{H}$ are monotone and Lipschitz continuous, they are maximally monotone operators with full domain, and so are $F_1$, \dots, $F_n$. According to \cite[Corollary 25.5(i)]{BC17}, $\mathcal{S}$ is a maximally monotone operator.

On the other hand, it follows from the definition of $\mathbf{y}^k$ that
\begin{align*}
\frac{1}{\gamma} D\mathbf{y}^k - \frac{1}{\gamma} N \mathbf{x}^k = \frac{1}{\gamma} M\mathbf{z}^k\in \ran(M\otimes \Id) =\Delta^\perp, 
\end{align*}
which yields  
\begin{align*}
\frac{1}{\gamma}\sum_{i=1}^n \delta_iy_i^k -\frac{1}{\gamma}\sum_{i=1}^n N_{ij}x_i^k =0,
\end{align*}
and thus
\begin{align}\label{eq:nth}
\sum_{i=1}^n \frac{\delta_i}{\gamma}(y_i^k -x_i^k) -\sum_{i=1}^n (\mathbf{F}_i \mathbf{x}^k -F_i x^k_i) =\frac{1}{\gamma}\sum_{i=1}^n \left(\sum_{j=1}^n N_{ij} -\delta_i\right)x_i^k -\sum_{i=1}^n \left(\mathbf{F}_i \mathbf{x}^k -F_i x^k_i\right).    
\end{align}
We observe that, for each $i\in \{1, \dots, n\}$,
\begin{multline*}
\mathbf{F}_i \mathbf{x}^k -F_i x^k_i =\sum_{j=1}^{p} (P_{ij} -Q_{ij})\left(B_j\left(\sum_{l=1}^n R_{jl}x^k_l\right) -B_j\left(\sum_{l=1}^n R_{jl}x^k_i\right)\right) \\
+\sum_{j=1}^{p} Q_{ij}\left(B_j\left(\sum_{l=1}^n P_{lj}x^k_l\right) -B_j\left(\sum_{l=1}^n P_{lj}x^k_i\right)\right).    
\end{multline*}
In view of \ref{l:shadow_scvg}, as $k\to +\infty$, for any $l, i\in \{1, \dots, n\}$, 
\begin{align}\label{eq:xlxi}
x^k_l -x^k_i\to 0,    
\end{align}
and so
\begin{align*}
\left\|\sum_{l=1}^n R_{jl}x^k_l -\sum_{l=1}^n R_{jl}x^k_i\right\|\to 0 \text{~~and~~} \left\|\sum_{l=1}^n P_{lj}x^k_l -\sum_{l=1}^n P_{lj}x^k_i\right\|\to 0,   
\end{align*}
which together with the Lipschitz continuity of $B_1$, \dots, $B_p$ imply that
\begin{align}
\label{eq:Fi}
\mathbf{F}_i \mathbf{x}^k -F_i x^k_i \to 0.
\end{align}
By combining this with \eqref{eq:nth}, Assumption~\ref{a:stand}\ref{a:stand_N}, and again \ref{l:shadow_scvg}, we deduce that
\begin{align}\label{eq:sumto0}
\sum_{i=1}^n \frac{\delta_i}{\gamma}(y_i^k -x_i^k) - \sum_{i=1}^n (\mathbf{F}_i \mathbf{x}^k  - F_i x^k_i)\to 0 \text{~~as~~} k\to +\infty.   
\end{align}
    
Now, let $\bar{\mathbf{z}} =(\bar{z}_1, \dots, \bar{z}_m)\in \mathcal{H}^m$ be an arbitrary weak cluster point of $(\mathbf{z}^k)_{k\in \mathbb{N}}$. By \eqref{eq:xlxi}, there exists $\bar{\mathbf{x}} = (\bar{x}, \dots, \bar{x})\in \Delta$ such that $(\bar{\mathbf{z}}, \bar{\mathbf{x}})$ is a weak cluster point of $(\mathbf{z}^k, \mathbf{x}^k)_{k\in \mathbb{N}}$. Denote $\bar{\mathbf{y}} = D^{-1}( M\bar{\mathbf{z}} + N\bar{\mathbf{x}})$. As the graph of a maximally monotone operator is sequentially closed in the weak-strong topology \cite[Proposition~20.38(ii)]{BC17}, in view of \eqref{eq:xlxi}, \eqref{eq:Fi}, and \eqref{eq:sumto0}, taking the limit in \eqref{eq:inclu} along a subsequence of $(\mathbf{x}^k)_{k\in \mathbb{N}}$ which converges weakly to $\bar{\mathbf{x}}$ yields
\begin{align*}
\begin{bmatrix}
0 \\
0 \\
\vdots \\
0 \\
0
\end{bmatrix}
\in \mathcal{S}\left(\begin{bmatrix}
\frac{\delta_1}{\gamma}(\bar{y}_1 -\bar{x})\\
\frac{\delta_2}{\gamma}(\bar{y}_2 -\bar{x})\\
\vdots \\
\frac{\delta_{n-1}}{\gamma}(\bar{y}_{n-1} -\bar{x})\\
\bar{x}
\end{bmatrix}\right),
\end{align*}
which is equivalent to
\begin{align}\label{eq:limit}
\begin{aligned}
\frac{\delta_i}{\gamma}(\bar{y}_i -\bar{x})- F_i\bar{x} &\in A_i \bar{x} , \quad i\in \{1, \dots, n-1\} \\
\frac{\delta_n}{\gamma}(\bar{y}_n -\bar{x})- F_n \bar{x} =-\sum_{i=1}^{n-1} \frac{\delta_i}{\gamma}(\bar{y}_i -\bar{x})- F_n \bar{x} &\in A_n \bar{x}.    
\end{aligned}    
\end{align}
It follows that
\begin{align*}
\bar{x} = J_{\frac{\delta_i}{\gamma}A_i}\left(\bar{y}_i - \frac{\gamma}{\delta_i}F_i\bar{x}\right)= J_{\frac{\delta_i}{\gamma}A_i}\left(\bar{y}_i - \frac{\gamma}{\delta_i}\mathbf{F}_i\bar{\mathbf{x}}\right), \quad i\in\{1,\dots,n\},
\end{align*}
from which we have $\bar{\mathbf{x}} =S\bar{\mathbf{z}}$, and so $T\bar{\mathbf{z}} =\bar{\mathbf{z}} -\theta M^{\top}\bar{\mathbf{x}} =\bar{\mathbf{z}}$, which means $\bar{\mathbf{z}}\in \Fix T$. By \cite[Theorem~5.5]{BC17}, $(\mathbf{z}^k)_{k\in \mathbb{N}}$ converges weakly to $\bar{\mathbf{z}}$. 

Finally, let $\bar{\mathbf{x}}\in \mathcal{H}^n$ be an arbitrary weak cluster point of $(\mathbf{x}^k)_{k\in \mathbb{N}}$. Then $(\bar{\mathbf{z}}, \bar{\mathbf{x}})$ is a weak cluster point of $(\mathbf{z}^k, \mathbf{x}^k)_{k\in \mathbb{N}}$ and, by \eqref{eq:xlxi}, $\bar{\mathbf{x}} = (\bar{x}, \dots, \bar{x})\in \Delta$ for some $\bar{x} \in \mathcal{H}$. By the same argument as above, we also obtain \eqref{eq:limit}, which implies that $\bar{x}\in \zer(\sum_{i=1}^n A_i +\sum_{j=1}^p B_j)$ (since $\sum_{i=1}^{n} F_i \bar{x} =\sum_{i=1}^{n}\sum_{j=1}^p P_{ij}B_j \bar{x} =\sum_{j=1}^p B_j \bar{x}$ due to Assumption~\ref{a:stand}\ref{a:stand_P}) and also $\bar{x} =J_{\frac{\gamma}{\delta_1} A_1}(\bar{y}_1) =J_{\frac{\gamma}{\delta_1} A_1}(\frac{1}{\delta_1} \sum_{j=1}^m M_{1j}\bar{z}_j)$. As the latter does not depend on the choice of cluster points, $\mathbf{x}^k\rightharpoonup \bar{\mathbf{x}} =(\bar{x}, \dots, \bar{x})$, and the proof is complete.
\end{proof}

Our main results on the convergence of Algorithm~\ref{algo:full} are presented in the following theorems.

\begin{theorem}[Convergence without cocoercivity]
\label{t:cvg_wco}
Suppose that $A_1$, \dots, $A_n$ are maximally monotone, that $B_1, \dots, B_{p}$ are monotone and $\ell$-Lipschitz continuous, that $\zer(\sum_{i=1}^n A_i + \sum_{j=1}^{p} B_j) \ne \varnothing$, that Assumption~\ref{a:stand} holds, and that $Q^{\top}\bOne = \bOne$. Let $\tau =\|(P^{\top}-Q^{\top})(M^{\top})^\dag\|^2+\|(P^{\top}-R)(M^{\top})^\dag\|^2$ and let \( (\mathbf{z}^k)_{k\in\mathbb{N}}\) and \( (\mathbf{x}^k)_{k\in\mathbb{N}}\) be the sequences generated by Algorithm~\ref{algo:full} with $\gamma \in \left(0, \frac{1}{\ell\tau}\right)$ and $(\lambda_k)_{k\in\mathbb{N}}$ in $[0,1-\gamma\ell\tau]$ satisfying $\liminf_{k\rightarrow +\infty} \lambda_k(1-\frac{\lambda_k}{1-\gamma\ell\tau}) > 0$. Then, as $k\to +\infty$, the following hold:
\begin{enumerate}
\item\label{t:cvg_wco_zTz} 
$(\Id -T)\mathbf{z}^k\to 0$ and $\sum_{i=1}^n s_ix^k_i\to 0$ whenever $(s_1, \dots, s_n)\in \mathbb{R}^n$ with $\sum_{i=1}^n s_i =0$. Moreover, $\|\frac{1}{k+1}\sum_{t=0}^k (\Id -T)\mathbf{z}^t\| = O(\frac{1}{\sqrt{k}})$.
\item\label{t:cvg_wco_x}
\( \mathbf{z}^k \rightharpoonup \bar{\mathbf{z}} \in \Fix T \) and \( \mathbf{x}^k \rightharpoonup (\bar{x}, \dots, \bar{x}) \in \mathcal{H}^n \) with \( \bar{x} \in \zer(\sum_{i=1}^n A_i + \sum_{j=1}^{p} B_j) \) provided that the first rows of $N$, $P$, and $Q$ are zero.
\end{enumerate}
\end{theorem}
\begin{proof}
\ref{t:cvg_wco_zTz}: Since \( \zer(\sum_{i=1}^n A_i + \sum_{j=1}^{p} B_j) \ne \varnothing \), we have from Lemma \ref{l:fixzeros} that \( \Fix T \ne \varnothing \). In view of Lemma \ref{l:nonexpansive}\ref{Lipschitz}, $T$ is conically $\rho$-quasiaveraged with $\rho=\frac{\theta}{1-\gamma\ell\tau}$. Using Proposition~\ref{p:KM}, we obtain that $(\mathbf{z}^k)_{k\in \mathbb{N}}$ is Fej\'er monotone with respect to $\Fix T$ as well as $(\Id -T)\mathbf{z}^k \to 0$ and $\|\frac{1}{k+1}\sum_{t=0}^k (\Id -T)\mathbf{z}^t\| = O(\frac{1}{\sqrt{k}})$ as $k\to +\infty$. This together with Remark~\ref{r:OnAlgo}\ref{r:OnAlgo_xz} implies the convergence properties of $\sum_{i=1}^n s_i x_i^k$.

\ref{t:cvg_wco_x}: This follows from \ref{t:cvg_wco_zTz}, the Fej\'er monotonicity of $(\mathbf{z}^k)_{k\in \mathbb{N}}$ with respect to $\Fix T$, and 
Lemma~\ref{l:shadow}.
\end{proof}

\begin{theorem}[Convergence with cocoercivity]
\label{t:cvg_co}
Suppose that $A_1$, \dots, $A_n$ are maximally monotone, that $B_1, \dots, B_p$ are $\frac{1}{\ell}$-cocoercive, that \( \zer(\sum_{i=1}^n A_i + \sum_{j=1}^{p} B_j) \ne \varnothing \), and that Assumption~\ref{a:stand} holds. Let $\tau=\|(P^{\top}-R)(M^{\top})^\dag\|^2$ and let \( (\mathbf{z}^k)_{k\in\mathbb{N}}\) and \( (\mathbf{x}^k)_{k\in\mathbb{N}}\) be the sequences generated by Algorithm~\ref{algo:full} with $Q=0$, $\gamma \in \left( 0, \frac{2}{ \ell\tau} \right)$, and $(\lambda_k)_{k\in \mathbb{N}}$ in $[0,\frac{2-\gamma\ell\tau}{2}]$ satisfying $\sum_{k=0}^{+\infty} \lambda_k(1-\frac{2\lambda_k}{2-\gamma\ell\tau}) = +\infty$. Then, as $k\to +\infty$, the following hold:
\begin{enumerate}
\item\label{t:cvg_co_zTz} 
$(\Id -T)\mathbf{z}^k\to 0$ and $\sum_{i=1}^n s_ix^k_i\to 0$ whenever $(s_1, \dots, s_n)\in \mathbb{R}^n$ with $\sum_{i=1}^n s_i =0$. Moreover, $\|(\Id -T)\mathbf{z}^k\| =o(\frac{1}{\sqrt{k}})$ if $\liminf_{k\rightarrow+\infty} \lambda_k(1-\frac{2\lambda_k}{2-\gamma\ell\tau}) > 0$.
\item\label{t:cvg_co_z}
\( \mathbf{z}^k \rightharpoonup \bar{\mathbf{z}} \in \Fix T \). 
\item\label{t:cvg_co_Bx}
$\mathbf{B}(R\mathbf{x}^k) \rightarrow \mathbf{B}(R\bar{\mathbf{x}})$ with $\bar{\mathbf{x}} = S\bar{\mathbf{z}}$.
\item\label{t:cvg_co_x}
\( \mathbf{x}^k \rightharpoonup (\bar{x}, \dots, \bar{x}) \in \mathcal{H}^n \) with \( \bar{x} \in \zer(\sum_{i=1}^n A_i + \sum_{j=1}^{p} B_j) \) provided that the first rows of $N$ and $P$ are zero.
\end{enumerate}
\end{theorem}
\begin{proof}
\ref{t:cvg_co_zTz} \& \ref{t:cvg_co_z}: As \( \zer(\sum_{i=1}^n A_i + \sum_{j=1}^{p} B_j) \ne \varnothing \), it holds that \( \Fix T \ne \varnothing \) due to Lemma \ref{l:fixzeros}. We then have from Lemma \ref{l:nonexpansive} that $T$ is conically $\rho$-averaged with $\rho=\frac{2\theta}{2-\gamma\ell\tau}$ and from \cite[Proposition~2.9]{BDP22} that $(\mathbf{z}^k)_{k\in \mathbb{N}}$ is Fej\'er monotone with respect to $\Fix T$, \( (\Id -T)\mathbf{z}^k \to 0 \), $\mathbf{z}^k \rightharpoonup \bar{\mathbf{z}} \in \Fix T$, and if $\liminf_{k\rightarrow+\infty} \lambda_k(1-\frac{2\lambda_k}{2-\gamma\ell\tau}) > 0$, then $\|(\Id -T)\mathbf{z}^k\| =o(\frac{1}{\sqrt{k}})$. Now, the conclusions on $\sum_{i=1}^n s_ix^k_i$ follows from Remark~\ref{r:OnAlgo}\ref{r:OnAlgo_xz}.
    
\ref{t:cvg_co_Bx}: Set $U =(P^{\top}-R)(M^{\top})^\dag$. As $\bar{\mathbf{z}}\in \Fix T$, we have from Lemma~\ref{l:nonexpansive}\ref{cocoercive} and Assumption~\ref{a:stand}\ref{a:stand_neg_semidef} that, for all $k\in \mathbb{N}$,
\begin{multline*}
\|T\mathbf{z}^k -\bar{\mathbf{z}}\|^2 +\left(\frac{1 -\theta}{\theta} -\frac{\gamma\ell}{2\theta}\|U\|^2\right)\|(\Id -T)\mathbf{z}^k\|^2 \\
+\frac{\gamma\ell}{2\theta}\left\| U(\Id -T)\mathbf{z}^k +\frac{2\theta}{\ell}(\mathbf{B}(R\mathbf{x}^k) -\mathbf{B}(R\bar{\mathbf{x}}))\right\|^2 \leq \|\mathbf{z}^k -\bar{\mathbf{z}}\|^2,
\end{multline*}
which is equivalent to
\begin{multline*}
\left(\frac{1}{\theta} -\frac{\gamma\ell}{2\theta}\|U\|^2\right)\|(\Id -T)\mathbf{z}^k\|^2 \\
+\frac{\gamma\ell}{2\theta}\left\| U(\Id -T)\mathbf{z}^k +\frac{2\theta}{\ell}(\mathbf{B}(R\mathbf{x}^k) -\mathbf{B}(R\bar{\mathbf{x}}))\right\|^2 \leq 2\langle \mathbf{z}^k -\bar{\mathbf{z}}, (\Id -T)\mathbf{z}^k\rangle.
\end{multline*}
Letting $k\to +\infty$ and noting that $(\Id -T)\mathbf{z}^k\to 0$ and $(\mathbf{z}^k)_{k\in \mathbb{N}}$ is bounded, we get the conclusion.

\ref{t:cvg_co_x}: Since $Q=0$ and since $\frac{1}{\ell}$-cocoercive operators are also monotone and $\ell$-Lipschitz continuous, the conclusion follows directly from \ref{t:cvg_co_zTz}, the Fej\'er monotonicity of $(\mathbf{z}^k)_{k\in \mathbb{N}}$ with respect to $\Fix T$, and Lemma~\ref{l:shadow}.
\end{proof}

We note that while our Theorem~\ref{t:cvg_co}\ref{t:cvg_co_Bx} generalizes \cite[Theorem 3.7(c)]{AMTT23}, its proof is simpler and more comprehensive due to the use of an extra term in the inequality in Lemma~\ref{l:nonexpansive}\ref{cocoercive}.

\begin{remark}[Ranges of parameters $\gamma$ and $\lambda_k$]
\label{r:ranges}
\begin{enumerate}
\item
In view of Theorems~\ref{t:cvg_wco} and \ref{t:cvg_co}, once Assumption~\ref{a:stand} holds, if $B_1$, \dots, $B_p$ are monotone and $\ell$-Lipschitz continuous, $Q^{\top}\bOne = \bOne$, and $\tau =\|(P^{\top}-Q^{\top})(M^{\top})^\dag\|^2+\|(P^{\top}-R)(M^{\top})^\dag\|^2$, then the conditions for $\gamma$ and $\lambda_k$ are
\begin{align}\label{eq:ranges_Lip}
\gamma \in \left(0, \frac{1}{\ell\tau}\right) \text{~and~} \lambda_k \in \left[0, 1-\gamma\ell\tau\right] \text{~with~} \liminf_{k\rightarrow +\infty} \lambda_k(1-\frac{\lambda_k}{1-\gamma\ell\tau}) > 0;    
\end{align}
if $B_1$, \dots, $B_p$ are $\frac{1}{\ell}$-cocoercive, $Q =0$, and $\tau =\|(P^{\top}-R)(M^{\top})^\dag\|^2$, then 
\begin{align}\label{eq:ranges_coc}
\gamma \in \left( 0, \frac{2}{ \ell\tau} \right) \text{~and~} \lambda_k \in \left[0, \frac{2-\gamma\ell\tau}{2}\right] \text{~with~} \sum_{k=0}^{+\infty} \lambda_k(1-\frac{2\lambda_k}{2-\gamma\ell\tau}) = +\infty.     
\end{align}
Here, we note that the step size corresponding to $A_i$ is $\frac{\gamma}{\delta_i}$ rather than $\gamma$.

\item\label{r:ranges_tau} 
To obtain the upper bounds for $\gamma$ and $\lambda_k$, computing $\tau$ is necessary though it may have a closed form in some cases. For example, given $Q=0$ and a matrix $M\in\mathbb{R}^{n\times m}$ satisfying Assumption~\ref{a:stand}\ref{a:stand_kerM}, in the case when every row sum of the strictly upper triangular part of $M^{\top}$ is nonzero, i.e., for all $i\in \{1, \dots, n-1\}$, $s_i :=\sum_{j=i+1}^n M_{ji}\neq 0$, we can choose $P\in \mathbb{R}^{n\times p}$ and $R\in \mathbb{R}^{p\times n}$ such that
\begin{align*}
P_{ji} &= \begin{cases}
\frac{M_{ji}}{s_i} &\text{if~} i\in \{1, \dots, p\},\ j\in \{i+1, \dots, n\}, \\
0 &\text{otherwise}
\end{cases} \\
\text{and~~} R &= P^{\top} - UM^{\top},
\end{align*}
where $U =[\diag(\frac{1}{s_1}, \dots, \frac{1}{s_p}) \,|\, 0_{p\times (m-p)}] \in \mathbb{R}^{p\times m}$. Then $P$ is lower triangular with zeros on the diagonal, $P^{\top}\bOne =\bOne $, $R$ is lower triangular, and $R\bOne = P^{\top}\bOne - UM^{\top}\bOne = P^{\top}\bOne = \bOne$. It follows that Assumption~\ref{a:stand}\ref{a:stand_P}--\ref{a:stand_R} are satisfied and $\tau =\|U\|^2 = \max\left\{\frac{1}{s_1^2}, \dots, \frac{1}{s_p^2}\right\}$. As will be seen later, this occurs in \emph{Case~1} of Remark~\ref{r:smallcases}, \emph{Case~1} of Example~\ref{eg:ring_seq}, the cocoercive cases of Examples~\ref{eg:star}, and Remark~\ref{r:ps}.
    
In the case when $Q\neq0$, simple choices of $P$, $Q$, and $R$ with a closed-form expression for $\tau$ are not always straightforward to derive. Nevertheless, such situations do exist, as described in \emph{Case~2} of Example~\ref{eg:ring_seq} and the Lipschitz cases of Example~\ref{eg:star}. It is worth noting that even when $\tau$ does not have a closed-form expression, it only requires a one-time computation prior to executing the algorithm, thus having negligible impact on performance.

\item 
As mentioned in Remark~\ref{r:variant}, in the setting where $B_1$, \dots, $B_p$ are monotone and $\ell$-Lipschitz continuous (resp., $\frac{1}{\ell}$-cocoercive), if additionally \eqref{eq:semidef_L} (resp., \eqref{eq:semidef_C}) holds, then we can set $\tau =0$ in \eqref{eq:ranges_Lip} (resp., \eqref{eq:ranges_coc}) and adopt the convention that $\frac{1}{\tau} =+\infty$. However, $\gamma$ must still satisfy \eqref{eq:semidef_L} (resp., \eqref{eq:semidef_C}), and solving for $\gamma$ may require computing the smallest generalized eigenvalue of $2D - N - N^\top - MM^\top$ with respect to $(P-Q)(P^{\top}-Q^{\top}) + (P- R^{\top})(P^{\top}- R)$ (resp., $(P - R^\top)(P^\top - R)$). 

Now, we consider a special case of \eqref{eq:semidef_C}, which is simpler than \eqref{eq:semidef_L}. Recall the setup and choices in the first part of \ref{r:ranges_tau} and suppose further that $2D-N-N^\top -MM^\top = MM^\top$. Set $s :=\min_{i\in\{1,\dots,p\}} s_i^2$. Then \eqref{eq:semidef_C} is equivalent to
\begin{align*}
MM^\top - \frac{\gamma\ell}{2}MU^\top U M^\top = M\left( \Id  - \frac{\gamma\ell}{2}\diag\left(\frac{1}{s_1^2}, \dots, \frac{1}{s_p^2}, 0, \dots, 0\right)\right)M^\top \succeq 0 \iff \gamma \leq \frac{2s}{\ell},    
\end{align*}
which is slightly less restrictive than $\gamma < \frac{2}{\ell\|U\|^2} = \frac{2s}{\ell}$, obtained from \eqref{eq:ranges_coc} under Assumption~\ref{a:stand}\ref{a:stand_neg_semidef}. This situation arises in \cite[Equation (3.15)]{BCLN22} considered in \emph{Case~1} of Example~\ref{eg:ring_seq}, \cite[Equation (3.14)]{BCLN22} considered in \emph{Case~1} of Example~\ref{eg:star}, and \eqref{algo:ps}--\eqref{algo:ps2} in Remark~\ref{r:ps}.
\end{enumerate}
\end{remark}

\section{Weighted graph-based forward-backward algorithms}
\label{s:graphs}

In this section, we provide some special cases of Algorithm~\ref{algo:full} for solving the monotone inclusion problem \eqref{gen_prob}. We describe how to obtain various weighted graph-based algorithms from our framework. Some of the derived algorithms can be seen as extensions of known methods, while the others such as those with explicit formulas, flexible weights, and without cocoercivity assumptions seem to be new and promising.

\paragraph{Some notations from graph theory.} 
An undirected graph is a pair $G=(\mathcal{V}, \mathcal{E})$, where $\mathcal{V}=\{v_1,\dots, v_n\}$ is a set of \emph{nodes} or \emph{vertices}, and $\mathcal{E}\subseteq \{\{u, v\}: u, v\in\mathcal{V},\ u \neq v\}$ is a set of unordered pairs of distinct nodes, called \emph{edges}. The edge $\{u,v\}$ is said to be \emph{incident} on $u$ and on $v$ while $u$ and $v$ are \emph{adjacent}. For any two nodes $u,v \in \mathcal{V}$, a \emph{path from $u$ to $v$} is a sequence of nodes $(v_0, v_1,\dots,v_r)$ such that $v_0=u, v_r=v$, and $\{v_i,v_{i+1}\} \in\mathcal{E}$ for $0\leq i\leq r-1$. The graph $G$ is \emph{connected} if for any two distinct nodes $u,v\in\mathcal{V}$, there is a path from $u$ to $v$. An \emph{orientation} of an undirected graph is the assignment of a direction to each edge in the graph. The \emph{adjacency matrix} $\Adj(G) \in \mathbb{R}^{n\times n}$ of $G$ is a symmetric matrix such that $\Adj(G)_{ij} = 1$ if $v_i$ and $v_j$ are adjacent (i.e., $\{v_i, v_j\}\in\mathcal{E}$), and $0$ otherwise. We say that $G$ is a \emph{weighted} graph if $G=(\mathcal{V}, \mathcal{E})$ is an undirected graph, where each edge $\{v_i, v_j\}\in \mathcal{E}$ is associated with $w_{ij}=w_{ji}\in \mathbb{R}$. By setting $w_{ij} =0$ if $\{v_i, v_j\}\notin \mathcal{E}$, we obtain a symmetric matrix $W=(w_{ij})\in\mathbb{R}^{n\times n}$ with $n =|\mathcal{V}|$ and call it \emph{weight matrix}.
In the unit-weight case, i.e., $w_{ij}\in\{0,1\}$, the adjacency matrix $\Adj(G)$ can be considered as the weight matrix of $G$. 

Let $G=(\mathcal{V}, \mathcal{E})$ is a weighted graph with weight matrix $W=(w_{ij})\in\mathbb{R}^{n\times n}$. For each node $v_i\in\mathcal{V}$, the \emph{degree} $d_i$ of $v_i$ is the sum of the weights of the edges incident on $v_i$, i.e., $d_i = \sum_{j=1}^n w_{ij}$, which reduces to the number of edges incident on $v_i$ in the unit-weight setting. The \emph{degree matrix} of $G$ is $\Deg(G)=\diag(d_1$, \dots, $d_n)\in\mathbb{R}^{n\times n}$. The \emph{incidence matrix} of $G$ with orientation $\sigma$ is $\Inc(G^{\sigma})\in\mathbb{R}^{n \times q}$ given by
\begin{align}\label{eq:Inc}
\Inc(G^\sigma)_{ij}=
\begin{cases}
\sqrt{w_{e_j}} \quad &\text{~if edge $e_j$ leaves node $v_i$},\\
-\sqrt{w_{e_j}} &\text{~if edge $e_j$ enters node $v_i$},\\
0 &\text{~otherwise},
\end{cases}
\end{align}
where $w_{e_j}$ denotes the weight of edge $e_j$, which is equal to $1$ if the graph has unit-weights. The \emph{Laplacian matrix} of a weighted graph is defined by $\Lap(G)=\Deg(G)-W$. According to \cite[Proposition 18.3]{GQ20}, $\Lap(G)$ is a symmetric and positive semidefinite matrix and 
\begin{align}\label{eq:decompose}
\Lap(G) = \Inc(G^\sigma)\Inc(G^\sigma)^{\top} = \Deg(G) - W. 
\end{align}  
Moreover, $\Inc(G^\sigma)\Inc(G^\sigma)^{\top}$ is independent of the orientation of the graph $G$. Laplacian matrix is invariant under node reordering, up to a permutation similarity \cite[Section 2]{Rus94}.

\paragraph{A graph-based selection for coefficient matrices $M$, $N$, and $D$.}
\label{eg:graph}
From now on, let $G =(\mathcal{V},\mathcal{E})$ be a weighted connected graph with $\mathcal{V}=\{v_1,\dots,v_n\}$ ($|\mathcal{V}|=n \geq 2$), $\mathcal{E}=\{e_1,\dots,e_q\}\subseteq \mathcal{V}\times \mathcal{V}$ ($|\mathcal{E}|=q \geq n-1$), and symmetric weight matrix $W=(w_{ij})\in [0, +\infty)^{n\times n}$. Let $G'=(\mathcal{V}, \mathcal{E'})$ be a weighted connected subgraph of $G$ with $\mathcal{E'}\subseteq \mathcal{E}$, $n-1 \leq |\mathcal{E'}|=m \leq q$, and symmetric weight matrix $W'=(w_{ij}')\in\mathbb{R}^{n\times n}$ such that, for all $(v_i,v_j)\in \mathcal{E'}$,
\begin{align*}
w_{ij}' = w_{ji}' = \mu_{ij}^2 \leq w_{ij} \text{~~with~} \mu_{ij}\in [0, +\infty),
\end{align*}
i.e., the weight on each edge of $G'$ is at most the weight on the corresponding edge of $G$. 

Let $M = (M_{ij}) \in\mathbb{R}^{n\times m}$ be such that $MM^{\top} = \Lap(G')$. Then Assumption~\ref{a:stand}\ref{a:stand_kerM} is satisfied, since $\ker M^{\top}=\ker MM^{\top}=\ker\Lap(G')=\lspan\{\bOne\}$, where the last equality follows from the connectedness of $G'$ and \cite[Theorem 13.1.1]{GR01}. In view of \eqref{eq:decompose}, a possible choice for $M$ is
\begin{align}\label{eq:M}
M = \Inc(G'^{\sigma})    
\end{align}
for some orientation $\sigma$ of $G'$. We choose 
$\sigma$ to be the orientation of $G'$ that assigns to each edge $\{v_i, v_j\}\in\mathcal{E'}$ the directed edge $(v_i, v_j)$ if $i<j$, and $(v_j, v_i)$ if $j<i$. 

Let $N=(N_{ij})\in\mathbb{R}^{n\times n}$ be given by 
\begin{align}\label{eq:N}
N_{ij} =\begin{cases}
w_{ji} &\text{if~} i >j, \\
0 &\text{otherwise}
\end{cases}    
\end{align}
and let 
\begin{align}\label{eq:D}
D=\diag(\delta_1,\dots,\delta_n)=\frac{1}{2}\Deg(G)\in\mathbb{R}^{n\times n},
\end{align}
i.e., for all $i\in \{1, \dots, n\}$, $\delta_i =\frac{1}{2}d_i$. Then $N$ is a lower triangular matrix with zero diagonal and 
\begin{align*}
\sum_{i,j=1}^n N_{ij} = \frac{1}{2}\sum_{i,j=1}^n (N_{ij} +N_{ji}) = \frac{1}{2}\sum_{i=1}^n \left(\sum_{j=1}^n w_{ij}\right) = \frac{1}{2}\sum_{i=1}^n d_i = \sum_{i=1}^n \delta_i,    
\end{align*}
which satisfies Assumption~\ref{a:stand}\ref{a:stand_N}.

To verify Assumption~\ref{a:stand}\ref{a:stand_neg_semidef}, we first observe that 
\begin{align}
\label{eq:LapGG'}
2D-N-N^{\top}-MM^{\top}=\Deg(G)-W- \Lap(G')=\Lap(G)-\Lap(G').    
\end{align}
Take any $u =(u_1, \dots, u_n)^\top \in \mathbb{R}^n$. By \cite[Proposition 18.4]{GQ20}, $u^{\top}\Lap(G)u= \frac{1}{2}\sum_{(i,j)\in \mathcal{E}} w_{ij}(u_i -u_j)^2$. As a result, 
\begin{align*}
u^\top(\Lap(G)-\Lap(G'))u = \frac{1}{2}\sum_{(i,j)\in \mathcal{E}\setminus\mathcal{E'}} w_{ij}(u_i -u_j)^2+ \frac{1}{2}\sum_{(i,j)\in \mathcal{E'}} (w_{ij} -w_{ij}')(u_i -u_j)^2 \geq 0,
\end{align*}
which implies that $2D-N-N^{\top}-MM^{\top} \succeq 0$. Hence, Assumption~\ref{a:stand} is satisfied provided that the matrices $P$ and $R$ meet Assumption~\ref{a:stand}\ref{a:stand_P}--\ref{a:stand_R}, which will be specified later depending on the particular scenario considered. 

It is worth noting from \eqref{eq:LapGG'} that, if $G$ and $G'$ are the same graph with the same weights, then $2D-N-N^{\top}-MM^{\top} =0$; if $G$ and $G'$ are the same graph with $w_{ij} =2w'_{ij}$ whenever $\{v_i, v_j\}\in \mathcal{E}$, then $2D-N-N^{\top}-MM^{\top} =MM^{\top}$. In the remainder of this section, we use the above graph-based selection of $M$, $N$, and $D$ for all subsequent examples, specializing it to different graph topologies to derive concrete algorithms.

\begin{example}[Weighted graph-based forward-backward algorithms]
For each $i\in \{2, \dots, p+1\}$, let $h(i)\in\{1, \dots, i-1\}$, chosen so that $\{v_{h(i)}, v_i\}\in \mathcal{E}$ whenever possible. Let $R=(R_{ij})\in\mathbb{R}^{p\times n}$ with
\begin{align*}
    R_{ij} =\begin{cases}
    1 &~\text{if~} j =h(i+1), \\
    0 &~\text{otherwise}.
    \end{cases}
\end{align*}
Then $R$ is lower triangular and satisfies Assumption~\ref{a:stand}\ref{a:stand_R}.

\emph{Case~1}: $p=n-1$ and $B_1,\dots,B_p$ are cocoercive. Then, by letting $Q=0$ and $P$ as in \eqref{eq:PQR1}, Algorithm~\ref{algo:full} becomes 
\begin{align}\label{eq:graph-basedFB}
    \begin{cases}
    x_1^k &= J_{\frac{\gamma}{\delta_1} A_1}\left(\frac{1}{\delta_1}\sum_{j=1}^{m} M_{1j}z^k_j\right),\\
    x_i^k &= J_{\frac{\gamma}{\delta_i} A_i}\left(\frac{1}{\delta_i}\sum_{j=1}^{m}M_{ij}z_j^k + \frac{1}{\delta_i}\sum_{j=1}^n N_{ij}x_j^k -\frac{\gamma}{\delta_i} B_{i-1} x_{h(i)}^k\right), \quad i\in\{2,\dots, n\},\\
    z_i^{k+1}\!\!\!\!\! &= z_i^k - \lambda_k\sum_{j=1}^n M_{ji}x_j^k, \quad i\in\{1,\dots, m\}.
    \end{cases}
\end{align}
In turn, when the graphs $G$ and $G'$ are both unit-weighted and $m=|\mathcal{E'}|=n-1$, this algorithm reduces to the \emph{forward-backward algorithms devised by graphs} \cite[Algorithm 1]{ACL24}. 

\emph{Case~2}: $p=n-2$ and $B_1,\dots, B_p$ are monotone and Lipschitz continuous. Let $P$ be as in \eqref{eq:PQR1} and $Q$ as in \eqref{eq:PQR2}. Algorithm~\ref{algo:full} then reduces to
\begin{align}
    \begin{cases}
    x_1^k &= J_{\frac{\gamma}{\delta_1} A_1}\left(\frac{1}{\delta_1}\sum_{j=1}^{m} M_{1j}z^k_j\right),\\
    x_i^k &= J_{\frac{\gamma}{\delta_i} A_i}\left(\frac{1}{\delta_i}\sum_{j=1}^{m}M_{ij}z_j^k + \frac{1}{\delta_i}\sum_{j=1}^n N_{ij}x_j^k -\frac{\gamma}{\delta_i} B_{i-1} x_{h(i)}^k\right), \quad i\in\{2,\dots, n-1\},\\
    x_n^k &= J_{\frac{\gamma}{\delta_n} A_n}\left(\frac{1}{\delta_n}\sum_{j=1}^m M_{nj} z_j^k + \frac{1}{\delta_n} \sum_{j=1}^n N_{nj} x_j^k + \frac{\gamma}{\delta_n}\sum_{j=1}^{n-2}B_j x_{h(j+1)}^k - \frac{\gamma}{\delta_n}\sum_{j=1}^{n-2} B_j x_{j+1}^k\right), \\
    z_i^{k+1}\!\!\!\!\! &= z_i^k - \lambda_k\sum_{j=1}^n M_{ji}x_j^k, \quad i\in\{1,\dots, m\}.
    \end{cases}
\end{align}
\end{example}

\begin{remark}[Algorithms for $n\in \{2, 3\}$ and $p =1$]
\label{r:smallcases}
We consider the following two cases.

\emph{Case~1}: $n =2$, $p =1$, and $B_1 =B$ is cocoercive. Then $G$ and $G'$ both have $n=2$ nodes connected by a single edge $e_1 =\{v_1, v_2\}$, with weights $w_{12}$ and $\mu_{12}^2$, respectively. According to \eqref{eq:Inc}, \eqref{eq:M}, \eqref{eq:N}, and \eqref{eq:D}, 
\begin{align*}
    M =
    \begin{bmatrix}
        \mu_{12}\\
        -\mu_{12}
    \end{bmatrix}, \quad
    N = 
    \begin{bmatrix}
        0 & 0 \\
        w_{12} & 0
    \end{bmatrix}, \quad
    \text{and~}D = \frac{1}{2} w_{12}
    \begin{bmatrix}
        1 & 0\\
        0 & 1
    \end{bmatrix}.
\end{align*}
Since $B$ is cocoercive, we choose $Q =0$. By Assumption~\ref{a:stand}\ref{a:stand_P}--\ref{a:stand_R}, for the algorithm to be explicit (see also Remark~\ref{r:OnAlgo}\ref{r:OnAlgo_explicit}), we can only take $P =[1 \ 0]^\top$ and $R =[1 \ 0]$. Algorithm~\ref{algo:full} now reduces to a special case of \eqref{eq:graph-basedFB}, and by setting $\hat{z}_1^k =\frac{2}{w_{12}}\mu_{12} z_1^k$, $\hat{\gamma} =\frac{2\gamma}{w_{12}}$, and $\hat{\lambda}_k =\lambda_k\frac{2\mu_{12}^2}{w_{12}}$, we obtain 
\begin{align}
\begin{cases}
x_1^k &= J_{\hat{\gamma} A_1}(\hat{z}_1^k),\\
x_2^k &= J_{\hat{\gamma} A_2}(2x_1^k -\hat{z}_1^k -\hat{\gamma} Bx_1^k),\\
\hat{z}_1^{k+1}\!\!\!\!\! &=\hat{z}_1^k -\hat{\lambda}_k(x_1^k -x_2^k),
\end{cases}
\end{align}
which is the Davis--Yin algorithm. In this case, if $G$ and $G'$ have the same weights (i.e., $w_{12} =\mu_{12}^2$), then the conclusions of Theorem~\ref{t:cvg_co} hold provided that $\hat{\gamma} \in \left(0, \frac{4}{\ell}\right)$, $\hat{\lambda}_k \in \left[0, 2 -\frac{\hat{\gamma}\ell}{2}\right]$, and $\sum_{k=0}^{+\infty} \hat{\lambda}_k(1 -\frac{2\hat{\lambda}_k}{4 -\hat{\gamma}\ell}) = +\infty$. This recovers \cite[Corollary~4.2]{DP21}, thereby improving \cite[Theorem~2.1(1)]{DY17}. 

\emph{Case~2}: $n =3$, $p =1$, and $B_1 =B$ is monotone and Lipschitz continuous. In view of Remark~\ref{r:OnAlgo}\ref{r:OnAlgo_explicit}, to make the algorithm explicit and satisfy Assumption~\ref{a:stand}\ref{a:stand_P}--\ref{a:stand_R} as well as $Q^{\top}\bOne=\bOne $, there is only one choice for $P$, $Q$, and $R$, namely, $P =[0\ 1\ 0]^\top$, $Q =[0\ 0\ 1]^\top$, and $R =[1\ 0\ 0]$.
    \begin{figure}[ht]
    \captionsetup{skip=0pt}
        \begin{center}

        \begin{subfigure}{0.24\textwidth}
        \centering
        
        \label{fig:sub1}
    
        \begin{tikzpicture}
        \def\n{3}  
        \def\radius{1.3cm}  
        \tikzset{every edge quotes/.style={fill=white, font=\footnotesize, inner sep=1pt}}
        
        \foreach \i in {1,...,\n} {
            \node[draw, circle, minimum size=1cm] (\i) at ({90-360/\n * (\i-1)}:\radius) {$v_{\i}$};
        }
        \draw[->] (1) -- 
        node[sloped, above, inner sep=1pt] {$e_{1}$} 
        node[sloped, below, inner sep=1pt] {$\mu_{12}^2$}
        (2);
        \draw[->] (1) -- 
        node[sloped, above, inner sep=1pt] {$e_{2}$} 
        node[sloped, below, inner sep=1pt] {$\mu_{13}^2$}
        (3);
        \draw[->] (2) -- 
        node[sloped, above, inner sep=1pt] {$e_{3}$} 
        node[sloped, below, inner sep=1pt] {$\mu_{23}^2$}
        (3);
        \end{tikzpicture}
        \caption{Complete/ring}
        \end{subfigure}
        \begin{subfigure}{0.24\textwidth}
        \centering
        
    
        \begin{tikzpicture}
        \def\n{3}  
        \def\radius{1.3cm}  
        \tikzset{every edge quotes/.style={fill=white, font=\footnotesize, inner sep=1pt}}
        
        \foreach \i in {1,...,\n} {
            \node[draw, circle, minimum size=1cm] (\i) at ({90-360/\n * (\i-1)}:\radius) {$v_{\i}$};
        }
        \draw[->] (1) -- 
        node[sloped, above, inner sep=1pt, pos=0.4] {$e_{1}$} 
        node[sloped, below, inner sep=1pt, pos=0.4] {$\mu_{12}^2$}
        (2);
        \draw[->] (2) -- 
        node[sloped, above, inner sep=1pt] {$e_{2}$} 
        node[sloped, below, inner sep=1pt] {$\mu_{23}^2$}
        (3);
        \end{tikzpicture}
        \caption{Sequential}
        \end{subfigure}
        \begin{subfigure}{0.24\textwidth}
        \centering
        
    
        \begin{tikzpicture}
        \def\n{3}  
        \def\radius{1.3cm}  
        \tikzset{every edge quotes/.style={fill=white, font=\footnotesize, inner sep=1pt}}
        
        \foreach \i in {1,...,\n} {
            \node[draw, circle, minimum size=1cm] (\i) at ({90-360/\n * (\i-1)}:\radius) {$v_{\i}$};
        }
        \draw[->] (1) -- 
        node[sloped, above, inner sep=1pt, pos=0.6] {$e_{1}$} 
        node[sloped, below, inner sep=1pt, pos=0.6] {$\mu_{12}^2$}
        (2);
        \draw[->] (1) -- 
        node[sloped, above, inner sep=1pt, pos=0.6] {$e_{2}$} 
        node[sloped, below, inner sep=1pt, pos=0.6] {$\mu_{13}^2$}
        (3);
        \end{tikzpicture}
        \caption{First-centered star}
        \end{subfigure}
        \begin{subfigure}{0.24\textwidth}
        \centering
        
    
        \begin{tikzpicture}
        \def\n{3}  
        \def\radius{1.3cm}  
        \tikzset{every edge quotes/.style={fill=white, font=\footnotesize, inner sep=1pt}}
        
        \foreach \i in {1,...,\n} {
            \node[draw, circle, minimum size=1cm] (\i) at ({90-360/\n * (\i-1)}:\radius) {$v_{\i}$};
        }
        \draw[->] (1) -- 
        node[sloped, above, inner sep=1pt, pos=0.4] {$e_{1}$} 
        node[sloped, below, inner sep=1pt, pos=0.4] {$\mu_{13}^2$}
        (3);
        \draw[->] (2) -- 
        node[sloped, above, inner sep=1pt] {$e_{2}$} 
        node[sloped, below, inner sep=1pt] {$\mu_{23}^2$}
        (3);
        \end{tikzpicture}
        \caption{Last-centered star}
        \end{subfigure}
        \end{center}
        \caption{Possible 3-node connected graphs. Edge numbers, directions, and weights refer to $G'$}
        \label{fig:3-complete}
    \end{figure}
As $G$ and $G'$ are connected graphs with $n=3$ nodes, all possible such graphs are depicted in Figure~\ref{fig:3-complete}. There are thus 7 scenarios: If $G$ is the complete graph, which in this setting coincides with the ring graph, then $G'$ can be any of the 4 graphs in the figure; if $G$ is one of the 3 remaining graphs, then $G'$, as a subgraph of $G$, must coincide with $G$ (possibly with different weights). 

Let us first consider the scenario where $G$ and $G'$ are both the complete graph. By \eqref{eq:Inc}, \eqref{eq:M}, \eqref{eq:N}, and \eqref{eq:D}, 
\begin{align*}
    M =
    \begin{bmatrix}
        \mu_{12} & \mu_{13} & 0\\
        -\mu_{12} & 0 & \mu_{23}\\
        0 & -\mu_{13} & -\mu_{23}
    \end{bmatrix},\
    N = 
    \begin{bmatrix}
        0 & 0 & 0\\
        w_{12} & 0 & 0\\
        w_{13} & w_{23} & 0
    \end{bmatrix},\ 
    D = \frac{1}{2}
    \begin{bmatrix}
        w_{12} + w_{13} & 0 & 0\\
        0 & w_{12} + w_{23} & 0\\
        0 & 0 & w_{13} + w_{23}
    \end{bmatrix}.
\end{align*}
Then, Algorithm~\ref{algo:full} becomes
\begin{align}\label{algo:3-complete}
    \begin{cases}
        x_1^k &= J_{\frac{2\gamma}{w_{12}+w_{13}} A_1}\Big(\frac{2}{w_{12} + w_{13}}(\mu_{12}z_1^k + \mu_{13}z_2^k)\Big),\\
        x_2^k &= J_{\frac{2\gamma}{w_{12} + w_{23}} A_2}\Big(\frac{2}{w_{12} + w_{23}}(w_{12} x_1^k - \mu_{12}z_1^k + \mu_{23} z_3^k  - \gamma B x_1^k)\Big),\\
        x_3^k &= J_{\frac{2\gamma}{w_{13} + w_{23}} A_3}\Big(\frac{2}{w_{13} + w_{23}}(w_{13}x_1^k + w_{23}x_2^k - \mu_{13}z_2^k - \mu_{23}z_3^k + \gamma B x_1^k - \gamma B x_2^k)\Big),\\
        z_1^{k+1}\!\!\!\!\! &= z_1^k - \lambda_k \mu_{12} (x_1^k - x_2^k),\\
        z_2^{k+1}\!\!\!\!\! &= z_2^k - \lambda_k \mu_{13}(x_1^k - x_3^k),\\
        z_3^{k+1}\!\!\!\!\! &= z_3^k - \lambda_k \mu_{23}(x_2^k - x_3^k).
    \end{cases}
\end{align}
The algorithms corresponding to the 6 remaining scenarios of $G$ and $G'$ can be derived from \eqref{algo:3-complete} by setting either one $\mu_{ij} =0$ (which removes the column containing $\mu_{ij}$ in $M$ and eliminates the calculation of $z_\ell^{k+1}$ that involves $\mu_{ij}$) or one $w_{ij} =0$ (which also implies $\mu_{ij} =0$).

Now, let $\mu_{13}=0$ (i.e., $G'$ is the sequential graph) and, for simplicity, let $\mu_{12}^2 =w_{12} =\mu_{23}^2 =w_{23} =1$. Then \eqref{algo:3-complete} becomes
\begin{align}\label{algo:3-ring_seq}
    \begin{cases}
        x_1^k &= J_{\frac{2\gamma}{1 +w_{13}} A_1}\Big(\frac{2}{1 +w_{13}}z_1^k\Big), \\
        x_2^k &= J_{\gamma A_2}(x_1^k -z_1^k +z_3^k -\gamma Bx_1^k), \\
        x_3^k &= J_{\frac{2\gamma}{1 +w_{13}} A_3}\Big(\frac{2}{1 +w_{13}}(w_{13}x_1^k +x_2^k -z_3^k +\gamma Bx_1^k -\gamma Bx_2^k)\Big), \\
        z_1^{k+1}\!\!\!\!\! &= z_1^k -\lambda_k(x_1^k -x_2^k), \\
        z_3^{k+1}\!\!\!\!\! &= z_3^k -\lambda_k(x_2^k -x_3^k).
    \end{cases}
\end{align}
In this case, the conclusions of Theorem~\ref{t:cvg_wco} hold whenever $\gamma \in \left(0, \frac{1}{2\ell}\right)$, $\lambda_k \in [0, 1-2\gamma\ell]$, and $\liminf_{k\rightarrow +\infty} \lambda_k(1-\frac{\lambda_k}{1-2\gamma\ell}) > 0$.

Setting $A_1 =0$ and $\lambda_k \equiv \frac{1 +w_{13}}{2}$, we have $x_1^k =\frac{2}{1 +w_{13}}z_1^k$ and $z_1^{k+1} =\frac{1 +w_{13}}{2}x_2^k$, so $x_1^k =x_2^{k-1}$. Thus, \eqref{algo:3-ring_seq} reduces to
\begin{align}\label{algo:2-ring_seq}
    \begin{cases}
        x_2^k &= J_{\gamma A_2}\Big(\frac{1 -w_{13}}{2}x_2^{k-1} +z_3^k -\gamma Bx_2^{k-1}\Big), \\
        x_3^k &= J_{\frac{2\gamma}{1 +w_{13}} A_3}\Big(\frac{2}{1 +w_{13}}(w_{13}x_2^{k-1} +x_2^k -z_3^k +\gamma Bx_2^{k-1} -\gamma Bx_2^k)\Big), \\
        z_3^{k+1}\!\!\!\!\! &= z_3^k -\frac{1 +w_{13}}{2}(x_2^k -x_3^k),
    \end{cases}
\end{align}
whose convergence requires $\lambda_k =\frac{1+w_{13}}{2}\leq 1 -2\gamma\ell$, which is equivalent to $\gamma \in \left(0, \frac{1 -w_{13}}{4\ell}\right]$. Further setting $A_3 =0$, we obtain that $x_3^k =\frac{2}{1 +w_{13}}(w_{13}x_2^{k-1} +x_2^k -z_3^k +\gamma Bx_2^{k-1} -\gamma Bx_2^k)$ and $z_3^{k+1} =(w_{13}x_2^{k-1} +x_2^k +\gamma Bx_2^{k-1} -\gamma Bx_2^k) -\frac{1 +w_{13}}{2}x_2^k$, which simplifies \eqref{algo:2-ring_seq} to   
\begin{align}
\label{algo:FRB}
x_2^k = J_{\gamma A_2}((1 -w_{13})x_2^{k-1} +w_{13}x_2^{k-2} -2\gamma Bx_2^{k-1} +\gamma Bx_2^{k-2}).
\end{align}
This reduces to the algorithm in \cite[Remark 7]{AMTT23} when $w_{13} =1$ (corresponding to the limiting case $\lambda_k \equiv 1$), and to the \emph{forward-reflected-backward algorithm} \cite{MT20} when $w_{13} =0$ (i.e., when $G$ is the sequential graph). The latter requires only $\gamma \in (0, \frac{1}{2\ell})$, owing to the underlying Lyapunov analysis. We note that, in each iteration, \eqref{algo:2-ring_seq} and \eqref{algo:FRB} require only one forward evaluation, as the remaining ones can be reused from the previous iteration. Motivated by the graph structures in Figure~\ref{fig:3-complete}, we extend the analysis to a general setting with arbitrary $n$.
\end{remark}

\begin{example}[Algorithms based on ring and sequential graphs]
\label{eg:ring_seq}
We consider $G=(\mathcal{V},\mathcal{E})$ to be a weighted ring graph with $n$ nodes and $q=n$ edges. Specifically, $\mathcal{E}=\{\{v_i,v_{i+1}\}: i\in\{1,\dots,n-1\}\} \cup \{\{v_1,v_n\}\}$, and for the weight matrix $W =(w_{ij})\in\mathbb{R}^{n\times n}$,
\begin{align*}
w_{ij} = w_{ji} = 0 \text{~~unless~~} (i,j)\in \{(i,i+1): i\in\{1,\dots,n-1\}\}\cup \{(1,n)\}.
\end{align*}
Then, by \eqref{eq:N} and \eqref{eq:D}, $N=(N_{ij})\in\mathbb{R}^{n\times n}$ and  $D=\diag(\delta_1,\dots,\delta_n)$ with
    \begin{align*}
    N_{ij}=
        \begin{cases}
            w_{ji}  &~\text{if }j\in\{1,\dots,n-1\},\ i=j+1,\\
            w_{1n} &~\text{if } i=n,\ j=1,\\
            0 & \text{~otherwise}
        \end{cases}
    \text{and }  
        \delta_i=\frac{1}{2}
        \begin{cases}
        w_{12}+w_{1n} &~\text{if } i=1,\\
        w_{i-1,i} + w_{i,i+1} &~\text{if }  i\in\{2,\dots,n-1\},\\
        w_{1n} + w_{n-1,n} &~\text{if } i=n.
         \end{cases}
    \end{align*}    

    \begin{figure}[ht!]
    \captionsetup{skip=0pt}
        \centering
        \begin{tikzpicture}[
            node distance=1.5cm,
        ]
    
        \node[draw, circle, minimum size=1cm] (1) {$v_1$};
        \node[draw, circle, minimum size=1cm,right=of 1] (2) {$v_2$};
        \node[draw, circle, minimum size=1cm,right=of 2] (3) {$v_3$};
        \node[draw, circle, minimum size=1cm,right=of 3] (dots) {$\dots$};
        \node[draw, circle, minimum size=1cm,right=of dots] (n-1) {$v_{n-1}$};
        \node[draw, circle, minimum size=1cm,right=of n-1] (n) {$v_n$};
    
        \draw[->] (1) edge 
        node[midway, above] {$e_1$}
        node[midway, below] {$\mu_{12}^2$}
        (2);
        \draw[->] (2) edge 
        node[midway, above] {$e_2$} 
        node[midway, below] {$\mu_{23}^2$}
        (3);
        \draw[->] (3) edge 
        node[midway, above] {$\dots$} 
        (dots);
        \draw[->] (dots) edge 
        node[midway, above] {$\dots$} 
        (n-1);
        \draw[->] (n-1) edge 
        node[midway, above] {$e_{n-1}$} 
        node[midway, below] {$\mu_{n-1,n}^2$}
        (n);
        \end{tikzpicture}
        \caption{A weighted sequential graph}
        \label{fig:sequential}
    \end{figure}
    Let $G'=(\mathcal{V},\mathcal{E'})$, where $\mathcal{E'}=\{e_i=\{v_i,v_{i+1}\} : i\in\{1,\dots,n-1\}\}$, be a weighted sequential graph. Then $m=|\mathcal{E'}|=n-1$. The weight matrix of $W'=(w'_{ij})\in\mathbb{R}^{n\times n}$ is given by
    \begin{align*}
        w'_{ij} = w'_{ji} = 
        \begin{cases}
            \mu_{ij}^2 \leq w_{ij} &~\text{if } i\in\{1,\dots,n-1\},\ j=i+1,\\
            0 &~\text{otherwise}.
        \end{cases}
    \end{align*}
    We number the edges of $G'$ as in Figure~\ref{fig:sequential}. By \eqref{eq:Inc} and \eqref{eq:M}, $M =(M_{ij})\in\mathbb{R}^{n\times (n-1)}$ with
    \begin{align}
        M_{ij} =
        \begin{cases}
            \mu_{ih} &~\text{if }i\in\{1,\dots,n-1\},\ h=i+1,\ j=i,\\ 
            -\mu_{hi} &~\text{if }i\in\{2,\dots,n\},\ h=i-1,\ j=i-1,\\ 
            0 &\text{~otherwise}.
        \end{cases}
        \label{formula:sequential}
    \end{align}

    \emph{Case~1}: $p=n-1$ and $B_1,\dots,B_{n-1}$ are cocoercive. Let $P$ and $R$ be as in \eqref{eq:PQR1}, and let $Q=0$. Then Algorithm~\ref{algo:full} becomes
    \begin{align}
    \label{algo:seq_FB}
        \begin{cases}
        x_1^k &= J_{\frac{\gamma}{\delta_1} A_1}\left(\frac{1}{\delta_1}\mu_{12} z_{1}^k\right), \\
        x_i^k &= J_{\frac{\gamma}{\delta_i} A_i}\left(\frac{1}{\delta_i}\Big(\mu_{i,i+1}z_i^k - \mu_{i-1,i} z_{i-1}^k + w_{i-1,i}x_{i-1}^k - \gamma B_{i-1}x_{i-1}^k \Big)\right),\ i\in\{2,\dots, n-1\},\\
        x_n^k &= J_{\frac{\gamma}{\delta_n} A_n}\left(\frac{1}{\delta_n}\Big(-\mu_{n-1,n}z_{n-1}^k + w_{1n} x_1^k + w_{n-1,n}x_{n-1}^k -\gamma B_{n-1} x_{n-1}^k\Big)\right),\\
        z_i^{k+1} \!\!\!\!\!&= z_i^k - \lambda_k \mu_{i,i+1}(x_i^k-x_{i+1}^k),\quad i\in\{1,\dots, n-1\},
        \end{cases}
    \end{align}
    which we call the \emph{weighted sequential forward-backward algorithm}. If $\mu_{i,i+1} =w_{i,i+1} =1$ for $i\in\{1,\dots, n-1\}$ and $w_{1n}=1$, then \eqref{algo:seq_FB} becomes the \emph{forward-backward method} for ring networks in \cite{AMTT23}. Especially, if $w_{1n} =0$ (i.e., both $G$ and $G'$ are weighted sequential graphs as in Figure~\ref{fig:sequential}), then we refer to the algorithm as the \emph{weighted sequential forward-Douglas--Rachford algorithm}, which reduces to the one in \cite[Equation (3.15)]{BCLN22} when $w_{i,i+1}=2$ and $\mu_{i,i+1}=1$ for $i\in\{1,\dots, n-1\}$.
    
    \emph{Case~2}: $p=n-2$ and $B_1,\dots,B_{n-2}$ are monotone and Lipschitz continuous. 
    We choose $P$, $Q$, and $R$ as in \eqref{eq:PQR1}. Then $P^{\top}-R = UM^{\top}$ and $P^{\top}-Q^{\top}= KM^{\top}$, where
\begin{align*}
    U = \left[ 
            \begin{array}{c|c} 
              -\diag(\frac{1}{\mu_{12}}, \frac{1}{\mu_{23}}, \dots, \frac{1}{\mu_{n-2,n-1}}) & 0_{(n-2)\times 1}
            \end{array} 
        \right], \, 
    K = \left[
            \begin{array}{c|c}
               0_{(n-2)\times 1}  & \diag(\frac{1}{\mu_{23}}, \frac{1}{\mu_{34}},\dots, \frac{1}{\mu_{n-1,n}})
            \end{array}
        \right].
\end{align*}
It follows that $\tau=\|U\|^2 + \|K\|^2 = \max_{1\leq i\leq n-2} \frac{1}{\mu_{i,i+1}^2} + \max_{2\leq i\leq n-1} \frac{1}{\mu_{i,i+1}^2}$.
In this case, Algorithm~\ref{algo:full} simplifies to
    \begin{align}
    \begin{cases}
    \label{algo:seq_FRB}
        x_1^k &= J_{\frac{\gamma}{\delta_1} A_1}\left(\frac{1}{\delta_1}\mu_{12} z_1^k\right), \\
        x_2^k &= J_{\frac{\gamma}{\delta_2} A_2}\left(\frac{1}{\delta_2}\Big(\mu_{23} z_2^k - \mu_{12}z_1^k + w_{12}x_1^k - \gamma B_1 x_1^k\Big)\right),\\
        x_i^k &= J_{\frac{\gamma}{\delta_i} A_i}\Bigg(\frac{1}{\delta_i}\Big(\mu_{i,i+1}z_i^k - \mu_{i-1,i}z_{i-1}^k + w_{i-1,i}x_{i-1}^k - \gamma (B_{i-1} x_{i-1}^k-B_{i-2}x_{i-2}^k)\\
        & \qquad\qquad\qquad-\gamma B_{i-2} x_{i-1}^k\Big)\Bigg), \quad i\in\{3,\dots, n-1\},\\
        x_n^k &= J_{\frac{\gamma}{\delta_n} A_n}\left(\frac{1}{\delta_n}\Big(-\mu_{n-1,n}z_{n-1}^k + w_{1n} x_1^k + w_{n-1,n}x_{n-1}^k +\gamma B_{n-2} x_{n-2}^k -\gamma B_{n-2} x_{n-1}^k\Big)\right),\\
        z_i^{k+1} \!\!\!\!\!&= z_i^k - \lambda_k \mu_{i,i+1}(x_i^k-x_{i+1}^k), \quad i\in\{1,\dots, n-1\},
    \end{cases}
    \end{align}
which is termed the \emph{weighted sequential forward-reflected-backward algorithm}. This, in turn, is the \emph{forward-reflected-backward method} in \cite{AMTT23} if $\mu_{i,i+1} =w_{i,i+1} =1$ for $i\in\{1,\dots, n-1\}$ and $w_{1n}=1$. 
\end{example}

\begin{example}[Algorithms based on star graphs]
\label{eg:star}
We consider $G =(\mathcal{V},\mathcal{E})$ to be a weighted star graph with $n$ nodes and $q=n-1$ edges. Let $G' =(\mathcal{V},\mathcal{E})$ be the same underlying graph with possibly different weights.

\emph{Case~1}: $v_1$ is the central node of $G$, i.e., $\mathcal{E} =\{\{v_1,v_j\}: j\in\{2,\dots,n\}\}$. Then the weight matrix $W =(w_{ij})\in\mathbb{R}^{n\times n}$ satisfies
\begin{align*}
w_{ij} = w_{ji} = 0 \text{~~unless~~} i=1,\ j\in\{2,\dots,n\}.
\end{align*}
In view of \eqref{eq:N} and \eqref{eq:D}, $N=(N_{ij})\in\mathbb{R}^{n\times n}$ and $D=\diag(\delta_1,\dots,\delta_n)\in \mathbb{R}^{n\times n}$ with
    \begin{align*}
    N_{ij}=
        \begin{cases}
            w_{ji} &~\text{if }j=1,\ i\in\{2,\dots,n\}, \\
            0 & \text{~otherwise},
        \end{cases}
    \text{ and } 
    \delta_i = \frac{1}{2}
    \begin{cases}
        \sum_{j=2}^n w_{1j}&~\text{if }i=1,\\
        w_{1i}&~\text{if }i\in\{2,\dots,n\}.
    \end{cases}
    \end{align*}    
We number the edges of $G'$ following an ordering as illustrated in Figure~\ref{fig:star1}. 
\begin{figure}[ht!]
\captionsetup{skip=0pt}
    \begin{center}
    \begin{subfigure}{0.45\textwidth}
    \centering
    \begin{tikzpicture}    
    \def\n{7}  
    \node[draw, circle, minimum size=1cm] (1) at (0,0) {$v_1$};

    \foreach \i in {2,3,4,5} {
        \node[draw, circle, minimum size=1cm] (\i) at ({90 - 360/(\n-1) * (\i-2)}:2.5cm) {$v_{\i}$};
        \pgfmathtruncatemacro{\k}{\i-1} 
        \draw[->] (1) -- 
        node[sloped, above] {$e_{\k}$} 
        node[sloped, below] {$\mu_{1\i}^2$}
        (\i);
    }

    \node[draw, circle, minimum size=1cm] (dots) at ({90 - 360/(\n-1) * (6-2)}:2.5cm) {$\dots$};
    \draw[->] (1) -- 
    node[sloped, above] {$\dots$} 
    (dots);

    \node[draw, circle, minimum size=1cm] (n) at ({90 - 360/(\n-1) * (7-2)}:2.5cm) {$v_n$};
    \draw[->] (1) -- 
    node[sloped, above] {$e_{n-1}$} 
    node[sloped, below] {$\mu_{1n}^2$}
    (n);
    \end{tikzpicture}
    \caption{First-centered star graph}
    \label{fig:star1}
    \end{subfigure}
    \begin{subfigure}{0.45\textwidth}
    \centering
    \begin{tikzpicture}    
    \def\n{7}  
    \node[draw, circle, minimum size=1cm] (n) at (0,0) {$v_n$};

    \foreach \i in {1,2,3,4} {
        \node[draw, circle, minimum size=1cm] (\i) at ({90 - 360/(\n-1) * (\i-1)}:2.5cm) {$v_{\i}$};
        \pgfmathtruncatemacro{\k}{\i} 
        \draw[<-] (n) -- 
        node[sloped, above] {$e_{\k}$} 
        node[sloped, below] {$\mu_{\k n}^2$} 
        (\i);
    }

    \node[draw, circle, minimum size=1cm] (dots) at ({90 - 360/(\n-1) * (6-2)}:2.5cm) {$\dots$};
    \draw[<-] (n) -- node[sloped, above] {$\dots$} (dots);

    \node[draw, circle, minimum size=1cm] (1) at ({90 - 360/(\n-1) * (7-2)}:2.5cm) {$v_{n-1}$};
    \draw[<-] (n) -- 
    node[sloped, above] {$e_{n-1}$} 
    node[sloped, below] {$\mu_{n-1,n}^2$} 
    (1);

    \end{tikzpicture}
    \caption{Last-centered star graph}
    \label{fig:star2}
    \end{subfigure}
    \end{center}   
    \caption{Weighted star graphs with $n$ nodes}
\end{figure}
According to \eqref{eq:Inc} and \eqref{eq:M}, $M =(M_{ij})\in\mathbb{R}^{n\times (n-1)}$ with
    \begin{align*}
        M_{ij} =
        \begin{cases}
            \mu_{ih} &~\text{if }i=1,\ 2\leq h \leq n,\ j=h-1, \\
            -\mu_{hi} &~\text{if }i\in\{2,\dots,n\},\ h=1,\ j=i-1,\\ 
            0 &\text{~otherwise}.
        \end{cases}
    \end{align*}

If $p=n-1$ and $B_1,\dots,B_{n-1}$ are cocoercive, then, defining $P$ as in \eqref{eq:PQR1}, $R$ as in \eqref{eq:PQR2}, and $Q=0$, Algorithm~\ref{algo:full} is expressed as
\begin{align}
\label{algo:par_up_cocoercive}
\begin{cases}
    x_1^k &= J_{\frac{2\gamma}{\sum_{j=2}^{n}w_{1j}} A_1}\left(\frac{2}{\sum_{j=2}^{n}w_{1j}}\sum_{h=2}^{n}\mu_{1h} z_{h-1}^k\right),\\
    x_i^k &= J_{\frac{2\gamma}{w_{1i}} A_i}\left(\frac{2}{w_{1i}}\Big(w_{1i} x_1^k- \mu_{1i} z_{i-1}^k - \gamma B_{i-1} x_1^k\Big)\right),\quad i\in\{2,\dots, n\},\\
    z_i^{k+1} \!\!\!\!\!&= z_i^k - \lambda_k\mu_{1,i+1}(x_1^k - x_{i+1}^k),\quad i\in\{1,\dots, n-1\},
\end{cases}
\end{align}
which is referred to as the \emph{weighted parallel up forward-Douglas--Rachford algorithm}. This becomes the algorithm proposed in \cite[Equation (3.14)]{BCLN22} when $w_{1j}=2$ and $\mu_{1j}=1$ for $j\in\{2,\dots, n\}$. If $B_1 =\dots =B_{n-1} =0$, then the latter reduces to the \emph{parallel Douglas--Rachford algorithm} (see \cite[Theorem 5.1]{Cam22}, \cite[Section 9.1]{CKMH23}), which is also known as a product-space reformulation with reduced dimension for the Douglas--Rachford algorithm.
    
If $p=n-2$ and $B_1,\dots,B_{n-2}$ are monotone and Lipschitz continuous, then, choosing $P$ as in \eqref{eq:PQR1}, $Q$ and $R$ as in \eqref{eq:PQR2}, Algorithm~\ref{algo:full} takes the form
\begin{align}
\label{algo:par_up_Lipschitz}
\begin{cases}
    x_1^k &= J_{\frac{2\gamma}{\sum_{j=2}^{n}w_{1j}} A_1}\left(\frac{2}{\sum_{j=2}^{n}w_{1j}}\sum_{h=2}^{n}\mu_{1h} z_{h-1}^k\right),\\
    x_i^k &= J_{\frac{2\gamma}{w_{1i}} A_i}\left(\frac{2}{w_{1i}}\Big(w_{1i} x_1^k- \mu_{1i} z_{i-1}^k - \gamma B_{i-1} x_1^k\Big)\right),\quad i\in\{2,\dots, n-1\},\\
    x_n^k &= J_{\frac{2\gamma}{w_{1n}}A_n}\left(\frac{2}{w_{1n}}\left(w_{1n}x_1^k - \mu_{1n}z_{n-1}^k + \gamma\sum_{j=1}^{n-2} B_j x_1^k - \gamma\sum_{j=1}^{n-2}B_j x_{j+1}^k\right)\right) ,\\
    z_i^{k+1} \!\!\!\!\!&= z_i^k - \lambda_k\mu_{1,i+1}(x_1^k - x_{i+1}^k),\quad i\in\{1,\dots, n-1\},
\end{cases}
\end{align}
which we name the \emph{weighted parallel up forward-aggregated-Douglas--Rachford algorithm}. By definition, $P^{\top}-R = UM^{\top}$ and $P^{\top}-Q^{\top}= KM^{\top}$, where
\begin{align*}
    U &= \left[
        \begin{array}{c|c}
        -\diag\left(\frac{1}{\mu_{12}}, \frac{1}{\mu_{13}}, \dots, \frac{1}{\mu_{1,n-1}}\right) & 0_{(n-2) \times 1}
        \end{array} 
        \right] \text{~and} \\ 
    K &= \left[
        \begin{array}{c|c}
            -\diag\left(\frac{1}{\mu_{12}}, \frac{1}{\mu_{13}}, \dots, \frac{1}{\mu_{1,n-1}}\right)  & \left(\frac{1}{\mu_{1n}}\right)_{(n-2)\times 1}
        \end{array}
        \right].
\end{align*}
When $\mu_{1n}=\cdots=\mu_{n-1,n}=\mu$, $\tau$ can be calculated explicitly as $\tau=\|U\|^2 + \|K\|^2 = \frac{1}{\mu^2} + \frac{n-1}{\mu^2} = \frac{n}{\mu^2}$.

\emph{Case~2}: $v_n$ is the central node of $G$, i.e., $\mathcal{E} =\{\{v_i,v_n\}:i\in\{1,\dots,n-1\}\}$. In this case, the edges of $G'$ are numbered as in Figure~\ref{fig:star2}. By \eqref{eq:Inc} and \eqref{eq:M}, $M =(M_{ij})\in\mathbb{R}^{n\times (n-1)}$ with
\begin{align}
    \label{formula:starM}
    M_{ij} =
    \begin{cases}
        \mu_{ih} &~\text{if }i\in\{1,\dots,n-1\},\ h=n,\ j=i, \\
        -\mu_{hi} &~\text{if }i=n,\ 1\leq h\leq n-1,\ j=h,\\ 
        0 &\text{~otherwise}.
    \end{cases}
\end{align}
Here, we omit the details of $N$ and $D$, which can be readily deduced from the algorithm. 

If $p=n-1$ and $B_1,\dots,B_p$ are cocoercive, then, defining $P$ as in \eqref{eq:PQR2} and $R$ as in \eqref{eq:PQR1}, we obtain 
\begin{align}
\label{algo:par_down_cocoercive}
\begin{cases}
    x_i^k &= J_{\frac{2\gamma}{\omega_{in}} A_{i}}\left(\frac{2}{\omega_{in}}\mu_{in}z_{i}^k\right),\quad i\in\{1,\dots, n-1\},\\
    x_{n}^k &= J_{\frac{2\gamma}{\sum_{j=1}^{n-1}\omega_{jn}}A_n}\left(\frac{2}{\sum_{j=1}^{n-1}\omega_{jn}}\Big(\sum_{j=1}^{n-1}\omega_{jn}x_j^k - \sum_{h=1}^{n-1} \mu_{hn}z_h^k - \gamma\sum_{i=1}^{n-1} B_i x_i^k\Big)\right),\\
    z_i^{k+1} \!\!\!\!\!&= z_i^k - \lambda_k\mu_{in}(x_i^k-x_{n}^k),\quad i\in\{1,\dots, n-1\},
\end{cases}
\end{align}
which is termed the \emph{weighted parallel down forward-Douglas--Rachford algorithm}.

If $p=n-2$ and $B_1,\dots,B_p$ are monotone and Lipschitz continuous, then, taking $P$ as in \eqref{eq:PQR1}, $Q$ and $R$ as in \eqref{eq:PQR2}, Algorithm~\ref{algo:full} becomes
\begin{align}
\label{algo:par_down_Lipschitz}
\begin{cases}
    x_1^k &= J_{\frac{2\gamma}{w_{1n}} A_{1}}\left(\frac{2}{w_{1n}}\mu_{1n}z_{1}^k\right),\\
    x_i^k &= J_{\frac{2\gamma}{w_{in}} A_{i}}\left(\frac{2}{w_{in}}(\mu_{in}z_{i}^k -\gamma B_{i-1} x_1^k )\right),\quad i\in\{2,\dots, n-1\},\\
    x_{n}^k &= J_{\frac{2\gamma}{\sum_{j=1}^{n-1}w_{jn}}A_n}\Bigg(\frac{2}{\sum_{j=1}^{n-1}w_{jn}}\Big(\sum_{j=1}^{n-1}w_{jn}x_j^k - \sum_{j=1}^{n-1} \mu_{jn}z_j^k \\
    &\hspace{5cm} + \gamma\sum_{j=1}^{n-2} B_j x_1^k - \gamma\sum_{j=1}^{n-2} B_j x_{j+1}^k\Big)\Bigg)\\
    z_i^{k+1} \!\!\!\!\!&= z_i^k - \lambda_k\mu_{in}(x_i^k-x_{n}^k),\quad i\in\{1,\dots, n-1\},
\end{cases}
\end{align}
which we refer to as the \emph{weighted parallel down forward-aggregated-Douglas--Rachford algorithm}.
\end{example}

\begin{remark}[Weighted product-space formulations of the Davis--Yin algorithm]
\label{r:ps}
We consider problem \eqref{gen_prob} with $p=n$ and $B_1,\dots, B_n$ cocoercive. A common approach is the product-space technique, which reformulates the problem as the three-operator inclusion
\begin{align*}
\text{find~} \mathbf{x} =(x_1, \dots, x_n)\in \mathcal{H}^n \text{~such that~} 0\in N_{\Delta}(\mathbf{x}) +\mathbf{A}(\mathbf{x}) +\mathbf{B}(\mathbf{x}),
\end{align*}
where $N_{\Delta}$ denotes the \emph{normal cone} to the diagonal subspace $\Delta$. We observe that $J_{\gamma \mathbf{A}}(\mathbf{x}) =(J_{\gamma A_1}(x_1), \dots, J_{\gamma A_n}(x_n))$ and $J_{N_{\Delta}}(\mathbf{x}) =(\frac{1}{n}\sum_{i=1}^n x_i, \dots, \frac{1}{n}\sum_{i=1}^n x_i)$. Applying the Davis--Yin algorithm to the operator triples $(N_{\Delta}, \mathbf{A}, \mathbf{B})$ and $(\mathbf{A},N_{\Delta}, \mathbf{B})$ yields the product-space formulations
\begin{align}
\label{algo:ps}
    &\begin{cases}
        x_1^k &= \frac{1}{n}\sum_{i=1}^n z_i^k,\\
        x_i^k &= J_{\gamma A_{i-1}}\Big(2x_1^k - z_{i-1}^k - \gamma B_{i-1}x_1^k\Big), \quad i\in\{2,\dots,n+1\},\\
        z_i^{k+1} \!\!\!\!\!&= z_i^k - \lambda_k(x_1^k-x_{i+1}^k), \quad i\in\{1,\dots,n\},
    \end{cases}\\
    \label{algo:ps2}
    \text{and~}&\begin{cases}
        x_i^k &= J_{\gamma A_i}(z_i^k), \quad i\in\{1,\dots,n\},\\
        x_{n+1}^k \!\!\!\!\!&= \frac{2}{n}\sum_{i=1}^n x_i^k - \frac{1}{n}\sum_{i=1}^n z_i^k - \frac{\gamma}{n}\sum_{i=1}^n B_i x_i^k,\\
        z_i^{k+1} \!\!\!\!\!&= z_i^k - \lambda_k(x_i^k-x_{n+1}^k), \quad i\in\{1,\dots,n\}.
    \end{cases}
\end{align}

Interestingly, algorithms \eqref{algo:ps} and \eqref{algo:ps2} can be derived from \eqref{algo:par_up_cocoercive} and \eqref{algo:par_down_cocoercive}, respectively. Indeed, using an artificial operator $A_0 = 0$ and applying algorithm \eqref{algo:par_up_cocoercive} to problem \eqref{gen_prob} with $(n+1)$ maximally monotone operators $A_0, A_1,\dots,A_n$ and $n$ cocoercive operators $B_1,\dots,B_n$, we obtain 
\begin{align}
    \begin{cases}
    x_1^k &= \frac{2}{\sum_{j=2}^{n+1}w_{1j}}\sum_{h=2}^{n+1}\mu_{1h} z_{h-1}^k,\\
    x_i^k &= J_{\frac{2\gamma}{w_{1i}} A_{i-1}}\left(\frac{2}{w_{1i}}\Big(w_{1i} x_1^k- \mu_{1i} z_{i-1}^k - \gamma B_{i-1} x_1^k\Big)\right),\quad i\in\{2,\dots, n+1\},\\
    z_i^{k+1} \!\!\!\!\!&= z_i^k - \lambda_k\mu_{1,i+1}(x_1^k - x_{i+1}^k),\quad i\in\{1,\dots, n\},
    \end{cases}
\end{align}
which reduces to \eqref{algo:ps} if $w_{1i}=2$ and $\mu_{1i}=1$ for $i\in\{2,\dots,n+1\}$. Note that if $B_1 = \dots =B_n :=\frac{1}{n}B$, then the latter becomes the \emph{generalized forward-backward algorithm} \cite{RFP13}, which recovers a product-space formulation of the Douglas--Rachford algorithm when $B=0$.

Similarly, introducing $A_{n+1}=0$ and applying \eqref{algo:par_down_cocoercive} to problem \eqref{gen_prob} with $(n+1)$ maximally monotone operators $A_1,\dots,A_n,A_{n+1}$ and $n$ cocoercive operators $B_1,\dots,B_n$, we derive the algorithm 
\begin{align}
\begin{cases}
    x_i^k &= J_{\frac{2\gamma}{w_{i,n+1}} A_{i}}\left(\frac{2}{w_{i,n+1}}\mu_{i,n+1}z_{i}^k\right),\quad i\in\{1,\dots, n\},\\
    x_{n+1}^k \!\!\!\!\!&= \frac{2}{\sum_{j=1}^{n}w_{j,n+1}}\Big(\sum_{j=1}^{n}w_{j,n+1}x_j^k - \sum_{h=1}^{n} \mu_{h,n+1}z_h^k - \gamma\sum_{i=1}^n B_i x_i^k\Big),\\
    z_i^{k+1} \!\!\!\!\!&= z_i^k - \lambda_k\mu_{i,n+1}(x_i^k-x_{n+1}^k),\quad i\in\{1,\dots, n\},
\end{cases}
\end{align}
which simplifies to \eqref{algo:ps2} if $w_{i,n+1}=2$ and $\mu_{i,n+1}=1$ for $i\in\{1,\dots,n\}$. Furthermore, if $B_1 =\dots =B_n =0$ for $i\in\{1,\dots,n\}$, then we obtain another product-space formulation of the Douglas--Rachford algorithm \cite[Proposition 26.12]{BC17}.
\end{remark}

\begin{example}[Algorithms based on complete and star graphs]
\label{eg:complete}
    Consider $G$ to be a weighted complete graph with $n$ nodes. Then $N=(N_{ij})\in\mathbb{R}^{n\times n}$ and $D=\diag(\delta_1,\dots,\delta_n)\in\mathbb{R}^{n\times n}$ with
    \begin{align*}
        N_{ij}=
        \begin{cases}
            w_{ji} &~\text{if } i>j,\\
            0   &~\text{otherwise}
        \end{cases}
        \quad \text{and} \quad
        \delta_i=\frac{1}{2}\Big(\sum_{j=i+1}^n w_{ij}+\sum_{j=1}^{i-1}w_{ji}\Big).
    \end{align*}
    Next, let $P\in \mathbb{R}^{n\times p}$ and $R\in \mathbb{R}^{p\times n}$ be any matrices satisfying Assumption~\ref{a:stand}\ref{a:stand_P}--\ref{a:stand_R}, and let $Q\in \mathbb{R}^{n\times p}$ be a matrix such that $Q =0$ if $B_1$, \dots, $B_p$ are cocoercive, and $Q^{\top}\bOne=\bOne $ if $B_1$, \dots, $B_p$ are monotone and Lipschitz continuous. Some simple selections of $P$, $Q$, and $R$ can be found in Remark~\ref{r:PQR}. For convenience, set $\mathbf{u}^k =(u_1^k, \dots, u_n^k) :=(P -Q)\mathbf{B}(R\mathbf{x}^k) -Q\mathbf{B}(P^{\top}\mathbf{x}^k)$.

    \emph{Case~1}: $G'=(\mathcal{V},\mathcal{E'})$ is also the complete graph $G$ with possibly different weights. Then $m=|\mathcal{E'}|=\frac{n(n-1)}{2}$. We number the edges of $G'$ following an ordering as illustrated in Figure~\ref{fig:complete}.
    \begin{figure}[ht]
    \captionsetup{skip=0pt}
        \begin{center}
        \begin{tikzpicture}
        \def\n{7}  
        \def\radius{3cm}  
\tikzset{every edge quotes/.style={fill=white, font=\footnotesize, inner sep=1pt}}

\foreach \i in {1,...,\n} {
    \node[draw, circle, minimum size=1cm] (\i) at ({90-360/\n * (\i-1)}:\radius) {$v_{\i}$};
}
        \draw[->] (1) edge["$e_1$", pos=0.4] (2);
        \draw[->] (1) edge["$e_2$", pos=0.15] (3);
        \draw[->] (1) edge["$e_3$", pos=0.1] (4);
        \draw[->] (1) edge["$e_4$", pos=0.1] (5);
        \draw[->] (1) edge["$e_5$", pos=0.15] (6);
        \draw[->] (1) edge["$e_6$", pos=0.4] (7);
        \draw[->] (2) edge["$e_7$", pos=0.4] (3);
        \draw[->] (2) edge["$e_8$", pos=0.15] (4);
        \draw[->] (2) edge["$e_9$", pos=0.1] (5);
        \draw[->] (2) edge["$e_{10}$", pos=0.1] (6);
        \draw[->] (2) edge["$e_{11}$", pos=0.1] (7);
        \draw[->] (3) edge["$e_{12}$", pos=0.4] (4);
        \draw[->] (3) edge["$e_{13}$", pos=0.15] (5);
        \draw[->] (3) edge["$e_{14}$", pos=0.15] (6);
        \draw[->] (3) edge["$e_{15}$", pos=0.1] (7);
        \draw[->] (4) edge["$e_{16}$", pos=0.4] (5);
        \draw[->] (4) edge["$e_{17}$", pos=0.15] (6);
        \draw[->] (4) edge["$e_{18}$", pos=0.1] (7);
        \draw[->] (5) edge["$e_{19}$", pos=0.4] (6);
        \draw[->] (5) edge["$e_{20}$", pos=0.15] (7);
        \draw[->] (6) edge["$e_{21}$", pos=0.4] (7);
        \end{tikzpicture}
        \end{center}
        \caption{A complete graph $G'$ of 7 nodes}
        \label{fig:complete}
    \end{figure}
    By a direct computation using \eqref{eq:Inc} and \eqref{eq:M}, $M =(M_{ij})\in\mathbb{R}^{n\times\frac{n(n-1)}{2}}$ is determined by
    \begin{align*}
        M_{ij} =
        \begin{cases}
            \mu_{ih} &~\text{if }i\in\{1,\dots,n-1\},\ i+1 \leq h \leq n,\ j=s(i)+h-i,\\ 
            -\mu_{hi} &~\text{if }i\in\{2,\dots,n\},\ 1 \leq h \leq i-1,\ j=s(h)-h+i, \\ 
            0 &\text{~otherwise},
        \end{cases}
    \end{align*}
    where $s(i)=\frac{(i-1)(2n-i)}{2}$. Algorithm~\ref{algo:full} takes the form
    \begin{align}
    \begin{cases}
        x_1^k &= J_{\frac{\gamma}{\delta_1}A_1}\left(\frac{1}{\delta_1}\Big(\sum_{h=2}^{n}\mu_{1h} z_{h-1}^k\Big)\right),\\
        x_i^k &= J_{\frac{\gamma}{\delta_i}A_i}\left(\frac{1}{\delta_i}\Big(\sum_{h=i+1}^{n} \mu_{ih} z_{s(i)+h-i}^k - \sum_{h=1}^{i-1} \mu_{hi} z_{s(h)-h+i}^k +\sum_{j=1}^{i-1}w_{ji}x_j^k -\gamma u_i^k\Big)\right),\\
        &\qquad\qquad\qquad i\in\{2,\dots, n-1\},\\
        x_n^k &= J_{\frac{\gamma}{\delta_n}A_n}\left(\frac{1}{\delta_n}\Big(-\sum_{h=1}^{n-1}\mu_{hn}z_{s(h)-h+n}^k + \sum_{j=1}^{n-1}w_{jn}x_j^k - \gamma u_n^k \Big)\right),\\
        z_h^{k+1} \!\!\!\!\!&= z_h^k - \lambda_k \mu_{i,-s(i)+h+i}\Big(x_i^k - x_{-s(i)+h+i}^k\Big), \quad i\in\{1,\dots,n-1\} \text{ and } s(i) < h \leq s(i+1).
    \end{cases}
    \end{align}
    While the latter seems complicated, Remark~\ref{r:OnAlgo}\ref{r:OnAlgo_reduced} provides a simplification by transforming $\mathbf{z}^k\in\mathcal{H}^{\frac{n(n-1)}{2}}$ into $\mathbf{v}^k= M\mathbf{z}^k \in\mathcal{H}^{n}$, thereby reducing the dimension from $\frac{n(n-1)}{2}$ to $n$ and yielding
    \begin{align*}
    \begin{cases}
    \mathbf{x}^k &=J_{\gamma D^{-1}\mathbf{A}}( D^{-1}(\mathbf{v}^k +  N\mathbf{x}^k -\gamma \mathbf{u}^k)), \\
    \mathbf{v}^{k+1} \!\!\!\!\!&=\mathbf{v}^k - \lambda_k M M^{\top}\mathbf{x}^k.
    \end{cases}
    \end{align*}
    By definition, $\Deg(G')=\diag(d_1',\dots, d_n')$ with $d_i'=\sum_{j=1}^n w_{ij}'=\sum_{j=i+1}^n \mu_{ij}^2 + \sum_{j=1}^{i-1} \mu_{ji}^2$, and so
    \begin{align*}
    MM^{\top} = \Lap(G') = \Deg(G')-W'
        =\begin{bmatrix}
            d_1' & -\mu_{12}^2 & -\mu_{13}^2 & \dots & -\mu_{1,n-1}^2 & -\mu_{1n}^2\\
            -\mu_{12}^2 & d_2' & -\mu_{23}^2 & \dots & -\mu_{2,n-1}^2 & -\mu_{2n}^2\\
            -\mu_{13}^2 & -\mu_{23}^2 & d_3' & \dots & -\mu_{3,n-1}^2 & -\mu_{3n}^2\\
            \vdots & \vdots & \vdots & \ddots & \vdots & \vdots\\
            -\mu_{1,n-1}^2 & -\mu_{2,n-1}^2 & -\mu_{3,n-1}^2 & \dots & d_{n-1}' & -\mu_{n-1,n}^2\\
            -\mu_{1n}^2 & -\mu_{2n}^2 & -\mu_{3n}^2 & \dots & -\mu_{n-1,n}^2 & d_n'
        \end{bmatrix}.
    \end{align*}
    Algorithm~\ref{algo:full} is then rewritten as
    \begin{align}
    \label{algo:complete}
    \begin{cases}
        x_1^k &= J_{\frac{\gamma}{\delta_1}A_1}\left(\frac{1}{\delta_1}v_1^k\right), \\
        x_i^k &= J_{\frac{\gamma}{\delta_i}A_i}\left(\frac{1}{\delta_i}\Big(v_{i}^k + \sum_{j=1}^{i-1}w_{ji}x_j^k - \gamma u_i^k\Big)\right),\quad i\in\{2,\dots, n\},\\
        v_i^{k+1}\!\!\!\!\! &= v_i^k - \lambda_k\left(d_i' x_i^k - \Big(\sum_{j=i+1}^n \mu_{ij}^2 x_j^k + \sum_{j=1}^{i-1} \mu_{ji}^2 x_j^k \Big)\right),\quad i\in\{1,\dots,n-1\},\\
        v_n^{k+1}\!\!\!\!\! &= v_n^k - \lambda_k\Big(d_n' x_n^k - \sum_{j=1}^{n-1}\mu_{jn}^2 x_j^k\Big).
    \end{cases}
    \end{align}
    
\emph{Case~2}: \label{eg:Ryu}
$G'=(\mathcal{V},\mathcal{E'})$ is a weighted star graph with $\mathcal{E'} =\{\{v_i,v_n\}:i\in\{1,\dots,n-1\}\}$, as in \emph{Case~2} of Example~\ref{eg:star}. The matrix $M$ is then given by \eqref{formula:starM} and Algorithm~\ref{algo:full} becomes 
\begin{align}
\label{algo:complete_star}
\begin{cases}
    x_1^k &= J_{\frac{\gamma}{\delta_1}A_1}\left(\frac{1}{\delta_1}\mu_{1n}z_1^k\right),\\
    x_i^k &= J_{\frac{\gamma}{\delta_i}A_i}\left(\frac{1}{\delta_i}\Big(\sum_{j=1}^{i-1}w_{ji}x_j^k + \mu_{in} z_i^k - \gamma u_i^k\Big)\right),\quad i\in\{2,\dots, n-1\},\\
    x_n^k &= J_{\frac{\gamma}{\delta_n }A_n}\left(\frac{1}{\delta_n}\Big(\sum_{j=1}^{n-1}w_{jn}x_j^k - \sum_{j=1}^{n-1}\mu_{jn} z_j^k-\gamma u_n^k\Big)\right), \\
    z_i^{k+1} \!\!\!\!\!&= z_i^k - \lambda_k\mu_{in}(x_i^k-x_{n}^k), \quad i\in\{1,\dots, n-1\},
\end{cases}
\end{align}
which can be viewed as a generalization of the Ryu splitting to the case of $n$ set-valued and $p$ single-valued operators. Indeed, it reduces to this scheme for $n$ set-valued operators (see \cite[Example 3.4]{Tam23}, \cite[Section 3.1]{BCN24}) when $p =0$, $w_{ij}=2$ for $i\neq j$, and $\mu_{in}=1$ for $i\in\{1,\dots, n-1\}$. 
\end{example}

\section{Numerical experiments}
\label{s:num_exp}
In this section, we provide numerical examples to test various choices of the coefficient matrices of Algorithm~\ref{algo:full}, specialized to graph-based algorithms in Section~\ref{s:graphs}. We also examine the influence of the graph structure, step size, and relaxation parameter on algorithm performance.

\subsection{Experiments with cocoercivity}
We first consider the problem from \cite[Section 5]{ACL24}, which takes the form
\begin{align}
\label{min_prob}
    \min_{x\in\mathbb{R}^{d}} \sum_{j=1}^{n-1}\left( \frac{1}{2} x^{\top} Q_j x\right) \text{~~subject to~~} x \in \bigcap_{i=1}^n C_i,
\end{align}
where $Q_j\in \mathbb{R}^{d\times d}$ is positive semidefinite for $j\in\{1,\dots,n-1\}$ and $C_i := \{ x\in\mathbb{R}^{d} \,|\, \|x - c_i\| \leq r_i\}$ for $i\in\{1,\dots,n\}$ have a common intersection point in the interior. This problem then can be formulated as an inclusion problem \eqref{gen_prob} as
\begin{align}
\label{num_prob}
    \text{find } x \in \mathbb{R}^{d} \text{~such that~} 0 \in \sum_{i=1}^n N_{C_i}x + \sum_{j=1}^{n-1} Q_jx,
\end{align}
where $N_{C_i}$, the normal cone to $C_i$, is maximally monotone, and $Q_j$ is $\frac{1}{\|Q_j\|_2}$-cocoercive. 

\paragraph{Data generated.} Random instances of problem \eqref{num_prob} are generated as in \cite[Section 5.1]{ACL24}, ensuring consistent problems with a nonempty feasible set that excludes the origin, which is the minimizer of the corresponding unconstrained problem. For each instance, random initial points for the algorithms are chosen outside the feasible set.

For comparison, we use the following methods: the \emph{weighted sequential forward-backward (FB) algorithm} \eqref{algo:seq_FB}, the \emph{weighted parallel up forward-Douglas--Rachford (FDR) algorithm} \eqref{algo:par_up_cocoercive}, the \emph{weighted parallel down forward-Douglas--Rachford (FDR) algorithm} \eqref{algo:par_down_cocoercive}, the \emph{complete graph algorithm with forward terms} \eqref{algo:complete}, and the \emph{complete-star graph algorithm with forward terms} \eqref{algo:complete_star}. For the last two methods, we test two sets of $P$ and $R$: one with $P$ and $R$ given by \eqref{eq:PQR1} and another with $P$ as in \eqref{eq:PQR1} and $R$ as in \eqref{eq:PQR2}, yielding a total of seven algorithms.

For simplicity, all underlying graphs are chosen with unit weights, and we examine the effect of varying the parameters $\gamma$ and $\lambda_k$. Algorithm performance is measured using the relative error 
\begin{align}\label{eq:relative error}
\max_{i\in\{1,\dots,n\}}\frac{\|x_i^k - x^*\|}{\|x^*\|}.
\end{align}
First, we select $\gamma = \hat{\gamma}\gamma_{\max}$, where $\hat{\gamma}\in(0,1)$ is a scaling factor and $\gamma_{\max} = \frac{2}{\ell\tau}$, with the cocoercive constant $\frac{1}{\ell}:=\min\left\{\frac{1}{\|Q_1\|_2}, \dots, \frac{1}{\|Q_{n-1}\|_2}\right\}$. Next, we chose another scaling factor $\hat{\lambda}\in(0,1)$ in which $\lambda_k = \lambda = \hat{\lambda}\lambda_{\max}$, where $\lambda_{\max} = \frac{2-\gamma\ell\tau}{2}$. 

For the first test, since $\lambda_{\max}$ depends on the choice of $\gamma$, we vary both parameters simultaneously to examine their combined effects. The results for the case when $n=50$ and $d=100$ of the complete graph 1 algorithm and the complete graph-star graph 1 algorithm are shown in Figure~\ref{fig:gamma_lambda}. 
\begin{figure}[htbp]
\captionsetup{skip=0pt}
    \centering
    \begin{subfigure}{0.48\textwidth}
        \includegraphics[width=\linewidth]{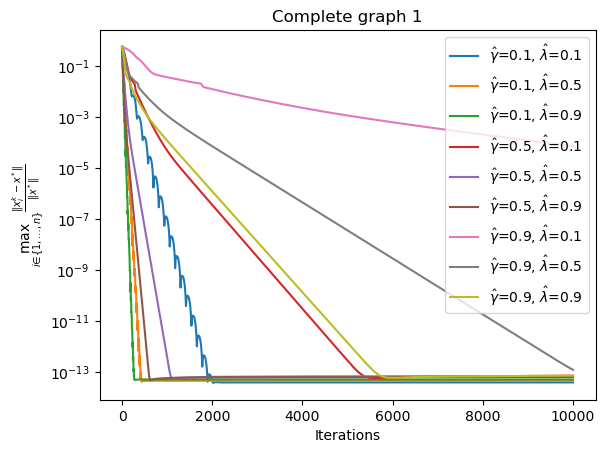}
    \end{subfigure}
    \begin{subfigure}{0.48\textwidth}
        \includegraphics[width=\linewidth]{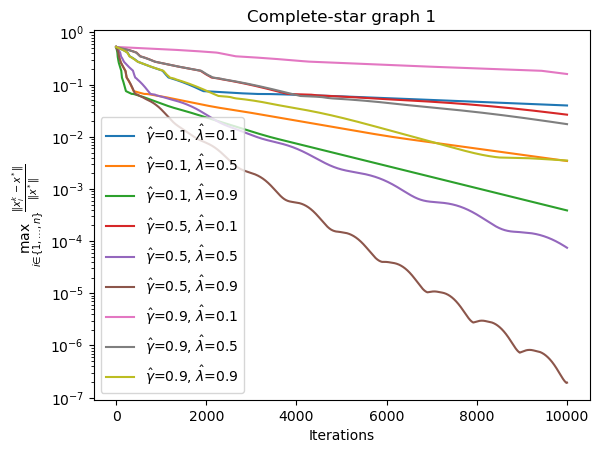}
    \end{subfigure}
    \caption{Impact of varying $\gamma$ and $\lambda_k$ on the performance.}
    \label{fig:gamma_lambda}
\end{figure}
The plots suggest that, while the complete-star graph 1 algorithm converges faster with $\gamma=0.5$, the complete graph 1 algorithm converges faster with smaller values of $\gamma$ in this scenario. However, they both perform better as $\lambda_k$ increase. We observe similar behavior for the complete-star graph 2 and complete graph 2 algorithms compared with the complete-star graph 1 algorithm, while the remaining algorithms behave similarly to the complete graph 1 algorithm. It is important to emphasize, however, that the optimal parameter values may vary depending on the problem setting.

In the second test, we evaluate the performance of several algorithms whose coefficient matrices are derived from different graph structures. For a fair comparison, we test feasible values of $\gamma$ and $\lambda_k$ and select those that give the best performance for each algorithm. The results for the cases where $n=50, d=100$ and where $n=100, d=100$ are reported in Figure~\ref{fig:cocoercive_comparison}.
\begin{figure}[htbp]
\captionsetup{skip=0pt}
    \centering
    \begin{subfigure}{0.48\textwidth}
        \includegraphics[width=\linewidth]{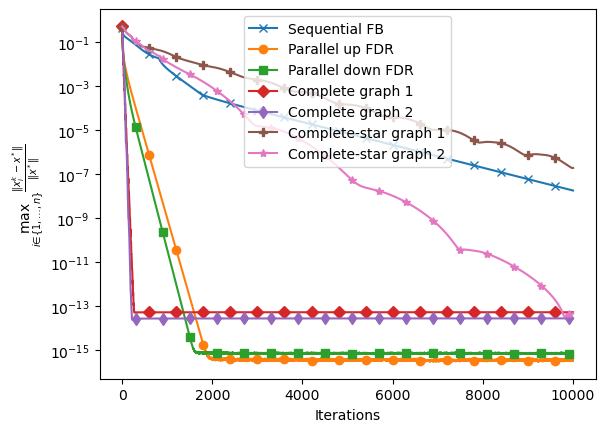}
    \end{subfigure}
    \begin{subfigure}{0.48\textwidth}
        \includegraphics[width=\linewidth]{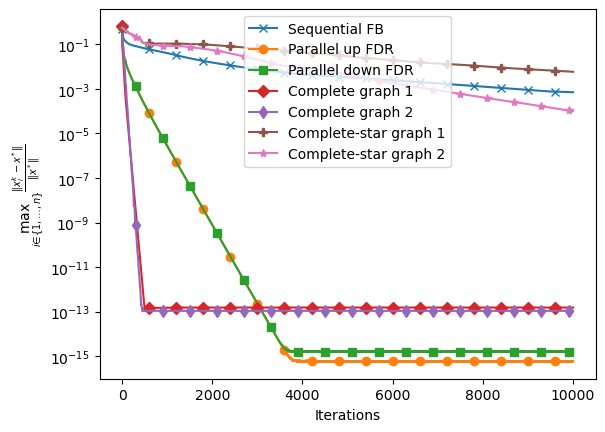}
    \end{subfigure}
    \caption{$n=50, d=100$ (left) and $n=100, d=100$ (right)}
    \label{fig:cocoercive_comparison}
\end{figure}
The results show that the complete graph 1 and 2 algorithms behave similarly and achieve the fastest convergence in terms of iteration count, although they require more runtime. The parallel up FDR and parallel down FDR algorithm also exhibit similar performance, forming the second group and ranking just behind the complete graph algorithms. The last group, which includes the sequential FB and the complete-star graph algorithms, performs not very well when $n=100$ and $d=100$. From Figure~\ref{fig:cocoercive_comparison}, one can also observe a large difference between the performance of the complete-star 1 and 2 algorithms when $n=50$, whereas the difference becomes less clear when $n=100$.

\subsection{Experiments without cocoercivity}
In the second numerical experiment, we study a zero-sum matrix game  between two teams, each consisting of $p$ players. The setting is the same as in \cite[Section 4.2]{TTZ25}. Team one has the payoff matrix $\Theta=\sum_{i=1}^p \Theta_i \in\mathbb{R}^{d_2 \times d_1}$, while team two has payoff matrix $-\Theta$. Each player in team one is paired with a player from team two, and pair $i$ is associated with the payoff matrices $(\Theta_i, -\Theta_i)$. The \emph{Nash equilibria} of this game correspond to the saddle points of the following min-max problem 
\begin{align}
    \label{prob:zero-sum}
    \min_{u\in\Delta^{d_1}}\max_{v\in\Delta^{d_2}} \sum_{j=1}^p \langle \Theta_j u, v \rangle \Longleftrightarrow  \min_{u\in\mathbb{R}^{d_1}}\max_{v\in\mathbb{R}^{d_2}} \sum_{i=1}^n  \iota_{\Delta^{d_1}}(u) + \sum_{j=1}^p \langle \Theta_j u, v \rangle - \sum_{i=1}^n  \iota_{\Delta^{d_2}}(v) 
\end{align}
where $\Delta^{d_1}=\{u=(u_1,\dots,u_{d_1})^\top \in [0, +\infty)^{d_1} \, | \, \sum_{j=1}^{d_1} u_j = 1\}$ is the unit simplex and $\Delta^{d_2}$ is defined similarly. In this setting, a saddle point is guaranteed to exist, and the subdifferential sum rule holds. Using the first-order optimality condition, solutions of the problem \eqref{prob:zero-sum} can be characterized as solutions of \eqref{gen_prob} with 
\begin{align}
    x =
    \begin{bmatrix}
        u\\
        v
    \end{bmatrix} \in \mathcal{H} = \mathbb{R}^{d_1} \times \mathbb{R}^{d_2}, \quad
    A_i = 
    \begin{bmatrix}
        N_{\Delta^{d_1}} & 0\\
        0 & N_{\Delta^{d_2}}
    \end{bmatrix}, \quad
    B_j = 
    \begin{bmatrix}
        0 & \Theta_j^{\top}\\
        -\Theta_j & 0
    \end{bmatrix}.
\end{align}
In the numerical experiments, we set $n=p+2$, $d_1=d_2=d$, and  choose the matrices $\Theta_j$ to be $M$-matrix, i.e.,
\begin{align}
    \Theta_j = s_j\Id - K_j, \quad s_j > \|K_j\|_2, \quad K_j > 0
\end{align}
where $K_j>0$ denotes a matrix with positive entries. This choice of matrices $\Theta_1,\dots,\Theta_p$ yields a \emph{completely mixed} game with a unique Nash equilibrium.

The initial point is $\mathbf{z}^0 = 0 \in \mathbb{R}^{m \times 2d}$ and the matrices $\Theta_j$ are generated as follows: first, generate $L_j\in\mathbb{R}^{d \times d}$ from the uniform distribution $U(0,1)$; second, take $K_j= j L_j$ and $s_j=1.1\|K_j\|_2$; then, let $\Theta_j=s_j \Id - K_j$. By definition, $B_j$ is monotone and $\|\Theta_j\|_2$- Lipschitz continuous.

We now compare the performance of seven algorithms which are the counterparts of the algorithms considered in the previous experiments: the \emph{weighted sequential forward-reflected-backward (FRB) algorithm} \eqref{algo:seq_FRB}, the \emph{weighted parallel up forward-aggregated-Douglas--Rachford algorithm} \eqref{algo:par_up_Lipschitz}, the \emph{weighted parallel down forward-aggregated-Douglas--Rachford algorithm} \eqref{algo:par_down_Lipschitz}, the \emph{complete graph algorithm with reflected terms} \eqref{algo:complete}, the \emph{complete-star graph algorithm with reflected terms} \eqref{algo:complete_star}. For the last two, we use two choices of $P$, $Q$, and $R$: one given by \eqref{eq:PQR1}, and another with $P$ as in \eqref{eq:PQR1} and $Q$, $R$ as in \eqref{eq:PQR2}.

For this problem, we also set the weights of all graphs as unit and measure performance using the relative error \eqref{eq:relative error}, as in the previous experiment. For each algorithm, optimal parameters $\gamma$ and $\lambda_k$ are selected from the feasible range. Figure~\ref{fig:Lipschitz_comparison} shows the results for the first experiment with $p=20$ and $d=50$ and the second experiment with $p=30$ and $d=50$ after 10,000 iterations. 
\begin{figure}[htbp]
\captionsetup{skip=0pt}
    \centering
    \begin{subfigure}{0.48\textwidth}
        \includegraphics[width=\linewidth]{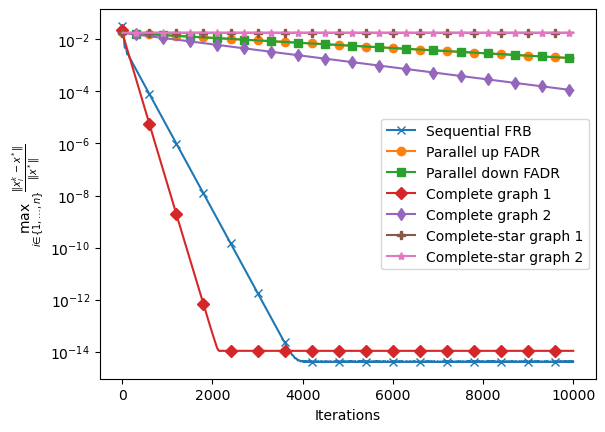}
    \end{subfigure}
    \begin{subfigure}{0.48\textwidth}
        \includegraphics[width=\linewidth]{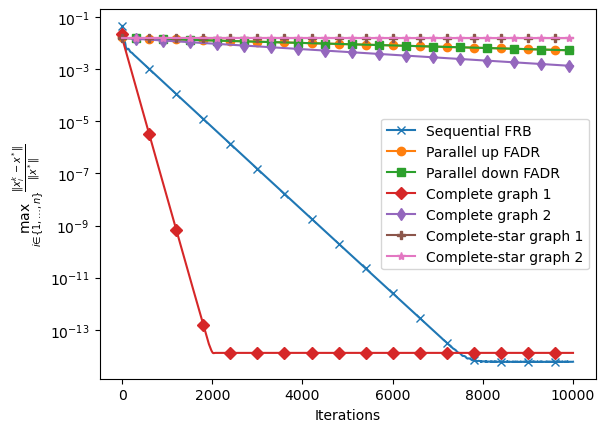}
    \end{subfigure}
    \caption{$p=20, d=50$ (left) and $p=30, d=50$ (right)}
    \label{fig:Lipschitz_comparison}
\end{figure}
The complete graph 1 algorithm converges the fastest for both $p=20$ and $p=30$. The sequential FRB algorithm also converges to high accuracy but requires more iterations when the problem size increases. In contrast, the parallel up/down FADR and complete graph 2 algorithms converge much more slowly, while the complete-star 1 and 2 algorithms show little improvement within 10,000 iterations. 

In summary, our experiments illustrate the framework's broad applicability and performance under diverse coefficient matrices and parameter settings. However, more comprehensive comparisons will be needed in future studies to draw general conclusions about algorithm performance.

\paragraph{Acknowledgements.} The research of MND, MKT and TDT was supported in part by Australian Research Council grant DP230101749. We thank the anonymous reviewers for their constructive comments, which helped improve the manuscript.


\begin{thebibliography}{99}

\setlength{\itemsep}{-2pt}

\bibitem{ACGN25}
A.~\r{A}kerman, E.~Chenchene, P.~Giselsson, and E.~Naldi, Splitting the forward-backward algorithm: A full characterization, \href{https://arxiv.org/abs/2504.10999}{arXiv:2504.10999}.

\bibitem{ACL24}
F.J.~Arag\'{o}n-Artacho, R.~Campoy, and C.~L\'{o}pez-Pastor, Forward-backward algorithms devised by graphs, \emph{SIAM J. Optim.}~{\bf 35}(4), 2423--2451.

\bibitem{AMTT23}
F.J.~Arag\'{o}n-Artacho, Y.~Malitsky, M.K.~Tam, and D.~Torregrosa-Bel\'{e}n, Distributed forward-backward methods for ring networks, \emph{Comput Optim Appl.}~{\bf 86}(3), 845--870 (2023).

\bibitem{APR18}
H.~Attouch, J.~Peypouquet, and P.~Redont, Backward-forward algorithms for structured
monotone inclusions in hilbert spaces, \emph{J. Math. Anal. Appl.}~{\bf 457}(2), 1095--1117 (2018).

\bibitem{BDP22}
S. Bartz, M.N. Dao, and H.M. Phan, 
Conical averagedness and convergence analysis of fixed point algorithms,
\emph{J. Glob. Optim.}~{\bf 82}(2), 351--373 (2022).

\bibitem{BC17}
H.H.~Bauschke and P.L.~Combettes,
\emph{Convex Analysis and Monotone Operator Theory in Hilbert Spaces}, 2nd ed., 
Springer, Cham (2017).


\bibitem{BCLN22}
K.~Bredies, E.~Chenchene, D.A.~Lorenz, and E.~Naldi, Degenerate preconditioned proximal point algorithms, \emph{SIAM J. Optim.}~{\bf 32}(3), 2376--2401 (2022).

\bibitem{BCN24}
K.~Bredies, E.~Chenchene, and E.~Naldi, Graph and distributed extensions of the Douglas–Rachford method, \emph{SIAM J. Optim.}~{\bf 34}(2), 1569--1594 (2024). 

\bibitem{Cam22}
R.~Campoy, A product space reformulation with reduced dimension for splitting algorithms, \emph{SIAM J. Control Optim.}~{\bf 83}(1), 319--348 (2022).

\bibitem{DP21}
M.N.~Dao and H.M.~Phan,
An adaptive splitting algorithm for the sum of two generalized monotone operators and one cocoercive operator, 
\emph{Fixed Point Theory Algorithm. Sci. Eng.}~{\bf 2021}(1) (2021).

\bibitem{CP11}
P.L.~Combettes and J.-C.~Pesquet, Proximal Splitting Methods in Signal Processing. In: H.H.~Bauschke, R.S.~Burachik, P.L.~Combettes, V.~Elser, D.R.~Luke, and H.~Wolkowicz (eds), \emph{Fixed-Point Algorithms for Inverse Problems in Science and Engineering}, 
Springer, New York, (2011). 

\bibitem{CKMH23}
L.~Condat, D.~Kitahara, A.~Contreras, and A.~Hirabayashi, Proximal splitting algorithms: A tour of recent advances, with new twists, \emph{SIAM Rev.}~{\bf 65}(2), 375--435 (2023).

\bibitem{DY17}
D.~Davis and W.~Yin, A three-operator splitting scheme and its optimization applications, \emph{Set-Valued Var. Anal.}~{\bf 25}(4), 829--858 (2017).


\bibitem{DR56}
J.~Douglas and H.H.~Rachford, On the numerical solution of heat conduction problems in two and three space variables, \emph{Trans. Am. Math. Soc.}~{\bf 82}(2), 421--439 (1956).


\bibitem{GS09}
G.~Strang, \emph{Introduction to Linear Algebra}, 4th edition, Wellesley (2009).

\bibitem{GQ20}
J.~Gallier and J.~Quaintance, \emph{Linear Algebra and Optimization with Applications to Machine Learning: Volume I: Linear Algebra for Computer Vision, Robotics, and Machine Learning} (2020).

\bibitem{GR01}
C.~Godsil and G.F.~Royle, \emph{Algebraic Graph Theory}, Graduate Texts in Mathematics.
Springer New York, NY (2001).

\bibitem{LM79}
P.L.~Lions and B.~Mercier, Splitting algorithms for the sum of two nonlinear operators, \emph{SIAM J. Nume. Anal.}~{\bf 16}(6), 964--979 (1979).

\bibitem{MT20}
Y.~Malitsky and M.K.~Tam, A forward-backward splitting method for monotone inclusions without cocoercivity, \emph{SIAM J. Optim.}~{\bf 30}(2), 1451--1472 (2020).

\bibitem{MT23}
Y.~Malitsky and M.K.~Tam, Resolvent splitting for sums of monotone operators with minimal lifting, \emph{Math. Program.}~{\bf 201}, 231--262 (2023).

\bibitem{Passty79}
G.B.~Passty, Ergodic convergence to a zero of the sum of monotone operators in Hilbert space, \emph{J. Math. Anal. Appl.}~{\bf 72}(2), 383--390 (1979).

\bibitem{Pen56}
R.~Penrose, On best approximate solution of linear matrix equations, \emph{Proc. Camb. Phil. Soc.}~{\bf 52}(1), 17--19 (1956).

\bibitem{RFP13}
H.~Raguet, J.~Fadili and G.~Peyr\'{e}, A generalized forward-backward splitting, \emph{SIAM J. Imaging Sci.}~{\bf 6}(3), 1199--1226 (2013).

\bibitem{RT20}
J.~Rieger and M.K.~Tam, Backward-forward-reflected-backward splitting for three operator monotone inclusions, \emph{Appl. Math. Comput.}~{\bf 381}, 125248 (2020).

\bibitem{Roc70}
R.T.~Rockafellar, Monotone operators associated with saddle-functions and minimax problems, In: F.E.~Browder (ed.) \emph{Nonlinear functional analysis, Proc. Symp. Pure Math.}~{\bf 18}, 241--250 (1970).

\bibitem{RW98}
R.T.~Rockafellar and R.~Wets, \emph{Variational Analysis}, vol.~317.~Springer, Berlin (1998). 

\bibitem{Rus94}
Russell Merris, Laplacian matrices of graphs: a survey, \emph{Linear Algebra Appl.}~{\bf 197--198}, 143--176 (1994).

\bibitem{Ryu20}
E.K.~Ryu, Uniqueness of DRS as the 2 operator resolvent-splitting and impossibility of 3 operator resolvent-splitting, \emph{Math. Program.}~{\bf 182}(1-2), 233--273 (2020).

\bibitem{RY22}
E.K.~Ryu and W.~Yin, \emph{Large-Scale Convex Optimization: Algorithm Designs via Monotone Operators}, Cambridge University Press, Cambridge (2022).

\bibitem{Sva11}
B.F.~Svaiter, On weak convergence of the Douglas–Rachford method, \emph{SIAM J. Control Optim.},~{\bf 49}(1), 280--287 
(2011).

\bibitem{Tam23}
M.K.~Tam, Frugal and decentralised resolvent splittings defined by nonexpansive operators, \emph{Optim Lett.}~{\bf 18}(7), 1541--1559 (2024).

\bibitem{TTZ25}
M.K.~Tam, L.~Timms, and L.~Zhang, A decentralised forward-backward-type algorithm with network-independent heterogeneous agent step sizes, \href{https://arxiv.org/abs/2512.12502}{arXiv:2512.12502}.

\bibitem{Tseng00}
P.~Tseng, A modified forward-backward splitting method for maximal monotone mappings, \emph{SIAM J. Control Optim.}~{\bf 38}(2), 431--446 (2000).

\end{thebibliography}
\end{document}